%% file: AutCoverCirculant.tex
\documentclass[11pt]{amsart}

\usepackage[utf8]{inputenc}

\usepackage{amssymb}
\usepackage{upref}
\RequirePackage{doi}
\usepackage{hyperref}
\usepackage[capitalize]{cleveref}
\usepackage{mathtools}

\numberwithin{equation}{section}

\makeatletter \renewcommand{\p@enumii}{} \makeatother

\newenvironment{thmref}[1]{\begingroup\def\theequation{\ref{#1}}}{\endgroup}

\newtheorem{cor}[equation]{Corollary}
\newtheorem{lem}[equation]{Lemma}
\newtheorem{prop}[equation]{Proposition}
\newtheorem{thm}[equation]{Theorem}

\crefformat{lem}{Lemma~#2\textup{#1}#3}
\Crefname{lem}{Lemma}{Lemmas}
\crefname{lem}{lemma}{lemmas}
\crefmultiformat{lem}{Lemmas~#2\textup{#1}#3}{ and~#2\textup{#1}#3}{, #2\textup{#1}#3}{ and~#2\textup{#1}#3}
\crefformat{cor}{Corollary~#2\textup{#1}#3}
\Crefname{cor}{Corollary}{Corollaries}
\crefname{cor}{corollary}{corollaries}
\crefformat{prop}{Proposition~#2\textup{#1}#3}
\Crefname{prop}{Proposition}{Propositions}
\crefname{prop}{proposition}{propositions}
\crefmultiformat{prop}{Propositions~#2\textup{#1}#3}{ and~#2\textup{#1}#3}{, #2\textup{#1}#3}{ and~#2\textup{#1}#3}
\crefformat{thm}{Theorem~#2\textup{#1}#3}
\Crefname{thm}{Theorem}{Theorems}
\crefname{thm}{theorem}{theorems}

\theoremstyle{definition}
\newtheorem{assump}[equation]{Assumption}

\newtheorem{defn}[equation]{Definition}
\newtheorem{eg}[equation]{Example}

\newtheorem{obs}[equation]{Observation}
\newtheorem{notation}[equation]{Notation}
\newtheorem*{notation*}{Notation}
\newtheorem{rem}[equation]{Remark}

\theoremstyle{remark}
\newtheorem*{ack}{Acknowledgments}
\newtheorem{rems}[equation]{Remarks}

\crefmultiformat{rem}{Remarks~#2\textup{#1}#3}{ and~#2\textup{#1}#3}{, #2\textup{#1}#3}{ and~#2\textup{#1}#3}
\crefformat{rem}{Remark~#2\textup{#1}#3}
\crefformat{rems}{Remark~#2\textup{#1}#3}
\Crefname{rem}{Remark}{Remarks}
\crefname{rem}{remark}{remarks}

\crefformat{figure}{Figure~#2\textup{#1}#3}

\newcounter{case}
\newenvironment{case}[1][\unskip]{\refstepcounter{case}\bf
\medskip \noindent Case \thecase\ #1.\ \it}{\unskip\upshape}
\numberwithin{case}{equation}
\renewcommand{\thecase}{\arabic{case}}
\crefformat{case}{Case~#2\textup{#1}#3}
\Crefname{case}{Case}{Cases}
\crefname{case}{case}{cases}

\newcounter{subcase}
\newenvironment{subcase}[1][\unskip]{\refstepcounter{subcase}\bf
\medskip \noindent Subcase \thesubcase\ #1.\ \it}{\unskip\upshape}
\numberwithin{subcase}{case}
\crefformat{subcase}{Subcase~#2\textup{#1}#3}
\crefname{subcase}{subcase}{subcases}
\Crefname{subcase}{Subcase}{Subcases}

\newcounter{subsubcase}

\numberwithin{subsubcase}{subcase}
\crefformat{subsubcase}{Subsubcase~#2\textup{#1}#3}
\crefname{subsubcase}{subsubcase}{subsubcases}
\Crefname{subsubcase}{Subsubcase}{Subsubcases}

\usepackage{color}
\newcommand{\refnote}[1]{\marginpar{%
	\color{blue}
	\vbox to 0pt{\vss
	$\begin{pmatrix} \text{see} \\[-3pt] \text{note} \\[-3pt] \text{\ref{#1}} \end{pmatrix}$%
	\vskip -1.1\baselineskip}}}

\theoremstyle{definition}
\newtheorem{aid}{}
\numberwithin{aid}{section}
\newcommand{\oldaid}{}
\let\oldaid=\aid
\renewcommand{\aid}{\bigbreak\oldaid}
\newcommand{\oldendaid}{}
\let\oldendaid=\endaid
\renewcommand{\endaid}{\oldendaid\bigskip\hrule width\textwidth \bigbreak}
\crefformat{aid}{Note~#2\textup{#1}#3}

\renewcommand{\pmod}[1]{\ (\mathrm{mod}~#1)}

\let\oldenumerate=\enumerate
\def\enumerate{\oldenumerate \itemsep=\smallskipamount}
\let\olditemize=\itemize
\def\itemize{\olditemize \itemsep=\smallskipamount}

\newcommand{\nh}{\mathord{\mathchoice
	{\hbox to 0pt{\vrule width 6.5pt height 2.75pt depth -1.95pt\hss}n}%
	{\hbox to 0pt{\vrule width 6.5pt height 2.75pt depth -1.95pt\hss}n}%
	{\hbox to 0pt{\vrule width 5pt height 2pt depth -1.6pt\hss}n}%
	{\hbox to 0pt{\vrule width 4.3pt height 1.6pt depth -1.1pt\hss}n}}}

\newcommand{\CC}{\mathbb{C}}
\newcommand{\QQ}{\mathbb{Q}}
\newcommand{\ZZ}{\mathbb{Z}}

\DeclareMathOperator{\Aut}{Aut}
\DeclareMathOperator{\Cay}{Cay}
\DeclareMathOperator{\lcm}{lcm}

\newcommand{\pref}[1]{\textup(\ref{#1}\textup)}
\newcommand{\fullref}[2]{\ref{#1}\pref{#1-#2}}
\newcommand{\fullcref}[2]{\cref{#1}\pref{#1-#2}}

\makeatletter
\newcommand{\noprelistbreak}{\smallskip\@nobreaktrue\nopagebreak} 
\makeatother

\begin{document}

\title[Automorphisms of the double cover of a circulant]{On automorphisms of the double cover \\ of a circulant graph}

\author{Ademir Hujdurović}
\address{University of Primorska, UP IAM, Muzejski trg 2, 6000 Koper, Slovenia and University of Primorska, UP FAMNIT, Glagolja\v ska 8, 6000 Koper, Slovenia}
\email{ademir.hujdurovic@upr.si}

\author{Đorđe Mitrović}
\address{University of Primorska, UP FAMNIT, Glagolja\v ska 8, 6000 Koper, Slovenia}
\email{mitrovic98djordje@gmail.com}

\author{Dave Witte Morris}
\address{Department of Mathematics and Computer Science, University of Lethbridge, Lethbridge, Alberta, T1K~3M4, Canada}
\email{dave.morris@uleth.ca}

\date{\today}

\begin{abstract}
A graph~$X$ is said to be \emph{unstable} if the direct product $X \times K_2$ (also called the canonical double cover of~$X$) has automorphisms that do not come from automorphisms of its factors~$X$ and~$K_2$. It is \emph{nontrivially unstable} if it is unstable, connected, and nonbipartite, and no two distinct vertices of X have exactly the same neighbors.

We find three new conditions that each imply a circulant graph is unstable. (These yield infinite families of nontrivially unstable circulant graphs that were not previously known.) We also find all of the nontrivially unstable circulant graphs of order~$2p$, where $p$ is any prime number. 

Our results imply that there does not exist a nontrivially unstable circulant graph of order~$n$ if and only if either $n$~is odd, or $n < 8$, or $n = 2p$, for some prime number~$p$ that is congruent to~$3$ modulo~$4$.
\end{abstract}

\maketitle

\section{Introduction}

Let $X$ be a circulant graph.
(All graphs in this paper are finite, simple, and undirected.)

\begin{defn}[\cite{CanCover}]
The \emph{canonical bipartite double cover} of~$X$ is the bipartite graph~$BX$ with $V(BX) = V(X) \times \{0,1\}$, where
	\[ \text{$(v,0)$ is adjacent to $(w,1)$ in $BX$}
	\quad \iff \quad
	\text{$v$ is adjacent to~$w$ in~$X$} . \]
\end{defn}

Letting $S_2$ be the symmetric group on the $2$-element set~$\{0,1\}$, it is clear that the direct product $\Aut X \times S_2$ is a subgroup of $\Aut BX$. We are interested in cases where this subgroup is proper:

\begin{defn}[{\cite[p.~160]{MarusicScapellatoSalvi}}]
 If $\Aut BX \neq \Aut X \times S_2$, then $X$ is \emph{unstable}.
\end{defn}

It is easy to see (and well known) that if $X$ is disconnected, or is bipartite, or has ``twin'' vertices (see \cref{TwinFreeDefn} below), 
then $X$ is unstable (unless $X$ is the trivial graph with only one vertex). The following definition rules out these trivial examples:

\begin{defn}[cf.\ {\cite[p.~360]{Wilson}}]
If $X$ is connected, nonbipartite, twin-free, and unstable, then $X$ is \emph{nontrivially unstable}. 
\end{defn}

S.\,Wilson found the following interesting conditions that force a circulant graph to be unstable. (See \cref{CayleyDefn} for the definition of the ``Cayley graph'' notation $\Cay(G, S)$.)

\begin{thm}[Wilson {\cite[Appendix~A.1]{Wilson} (and \cite[p.~156]{QinXiaZhou})}] \label{Wilson}
Let $X = \Cay(\ZZ_n, S)$ be a circulant graph, such that $n$ is even. Let $S_e = S \cap 2 \ZZ_n$ and $S_o = S \setminus S_e$. If any of the following conditions is true, then $X$ is unstable.
\noprelistbreak
	\begin{enumerate}
	\renewcommand{\theenumi}{C.\arabic{enumi}}
	\item \label{Wilson-C1}
	There is a nonzero element~$h$ of\/ $2 \ZZ_n$, such that $h + S_e = S_e$.
	
	\begingroup \renewcommand{\theenumi}{C.\arabic{enumi}$'$}
	\item \label{Wilson-C2}
	$n$~is divisible by~$4$, and there exists $h \in 1 + 2\ZZ_n$, such that
		\begin{enumerate}
		\item \label{Wilson-C2-So}
		$2h + S_o = S_o$, 
		and 
		\item \label{Wilson-C2-s+b}
		for each $s \in S$, such that $s \equiv 0$ or~$-h \pmod{4}$, we have $s + h \in S$.
		\end{enumerate}
	\endgroup
	
	\begingroup \renewcommand{\theenumi}{C.\arabic{enumi}$'$}
	\item \label{Wilson-C3}
	There is a subgroup~$H$ of\/~$\ZZ_n$, such that the set
		\[ R = \{\, s \in S \mid s + H \not\subseteq S \,\} ,\]
	is nonempty and has the property that if we let $d = \gcd \bigl( R \cup \{n\} \bigr)$, then $n/d$ is even, $r/d$ is odd for every $r \in R$, and either $H \nsubseteq d \ZZ_n$ or $H \subseteq 2d \ZZ_n$.
	\endgroup
	
	\item \label{Wilson-C4}
	There exists $m \in \ZZ_n^\times$, such that $(n/2) + mS = S$.
	\end{enumerate}
\end{thm}

\begin{rem} \label{WilsonCorrectionsRemark}
As will be explained in \cref{WilsonCorrections}, the statements~\pref{Wilson-C2} and~\pref{Wilson-C3} are slightly corrected versions of the original statements of Theorems~C.2 and~C.3 that appear in~\cite{Wilson}. The correction~\pref{Wilson-C2} is due to Qin-Xia-Zhou \cite[p.~156]{QinXiaZhou}.
\end{rem}

\begin{defn}
We say that $X$ has \emph{Wilson type} \pref{Wilson-C1}, \pref{Wilson-C2}, \pref{Wilson-C3}, or~\pref{Wilson-C4}, respectively, if it satisfies the corresponding condition of \cref{Wilson}. 
\end{defn}

In this terminology (modulo the corrections mentioned in \cref{WilsonCorrectionsRemark}), Wilson \cite[p.~377]{Wilson} conjectured that every nontrivially unstable circulant graph has a Wilson type. Unfortunately, this is not true, because counterexamples of order~24 were found by Qin-Xia-Zhou \cite[p.~156]{QinXiaZhou} (cf.\ \cref{Unstable24}).

Three of our results provide new conditions that force a circulant graph to be unstable. (These conditions provide infinitely many new counterexamples.) It seems likely that other conditions (and additional counterexamples) remain to be discovered.

Our first condition includes~\pref{Wilson-C1} as the special case where $K = \ZZ_n$, includes~\pref{Wilson-C2} as the special case where $K = 2\ZZ_n$, and includes~\pref{Wilson-C3} as another special case  (see~\cref{WilsonIsSpecialCase}).

\begin{thmref}{WilsonC1C3}
\begin{thm}[cf.\ {\cite[Thms.~1, C.1, and C.3]{Wilson}}] 
Let $X = \Cay(\ZZ_n, S)$ be a circulant graph.
Choose nontrivial subgroups~$H$ and~$K$ of\/~$\ZZ_n$, such that $|K|$ is even, and let $K_o = K \setminus 2K$. If either 
	\begin{enumerate}
	\item
	$S + H\subseteq S \cup (K_o + H)$ and $H \cap K_o = \emptyset$,
	or
	\item
	$(S \setminus K_o) + H \subseteq S \cup K_o$ and either\/ $|H| \neq 2$ or\/ $|K|$ is divisible by\/~$4$,
	\end{enumerate}
then $X$ is not stable.
\end{thm}
\end{thmref}

Our second condition is the following generalization of~\pref{Wilson-C4}.

\begin{thmref}{IsoTranslateS}
\begin{prop}
Assume $X = \Cay(\ZZ_n, S)$ is a circulant graph of even order. If $X \cong \Cay \bigl( \ZZ_n, (n/2) + S \bigr)$, then $X$ is unstable.
\end{prop}
\end{thmref}

Letting $S_e$ be the set of even elements of~$S$, and $S_o$ be the set of odd elements, as in \cref{Wilson}, our third condition states that if $\Cay(\ZZ_n, S_e)$ is unstable, and $S_o$ is invariant under sufficiently many translations, then $X$ is unstable.

\begin{thmref}{XeUnstable}
\begin{prop} 
Assume $X = \Cay(\ZZ_n, S)$ is a circulant graph of even order.
If there exist permutations $\alpha$ and~$\beta$ of\/~$2 \ZZ_n$, and a subgroup~$H$ of~$2\ZZ_n$, such that:
	\begin{enumerate}
	\item $\alpha \neq \beta$,
	\item if the vertices $u,v \in 2 \ZZ_n$ are adjacent, then the vertices $\alpha(u)$ and $\beta(v)$ are also adjacent,
	\item $w + H \subseteq S$, for all odd $w \in S$,
	and
	\item $\alpha(v) - v \in H$ and $\beta(v) - v \in H$, for all $v \in 2\ZZ_n$,
	\end{enumerate}
then $X$ is unstable.
\end{prop}
\end{thmref}

Wilson's conjecture is vacuously true for graphs of odd order:

\begin{thm}[Fernandez-Hujdurović \cite{FernandezHujdurovic} (or \cite{Morris-OddAbelian})] \label{OddCirculant}
There are no nontrivially unstable circulant graphs of odd order.
\end{thm}

Therefore, only graphs of even order are of interest. Another of our main results shows that the conjecture is true in the easiest of these interesting cases:

\begin{thmref}{2p}
\begin{thm}
If $p$ is a prime number, then every nontrivially unstable circulant graph of order~$2p$ has Wilson type~\pref{Wilson-C4}.
\end{thm}
\end{thmref}

This allows us to provide an explicit description of the nontrivially unstable circulant graphs of such orders:

\begin{thmref}{Describe2p}
\begin{cor}
Let $X = \Cay(\ZZ_{2p}, S)$ be a connected, twin-free, nonbipartite, circulant graph of order~$2p$, where $p$ is prime, and let $S_e = S \cap 2\ZZ_{2p}$. The graph~$X$ is nontrivially unstable if and only if there exists $m \in \ZZ_{2p}^\times$, such that $m^2 S_e = S_e$, $m S_e \neq S_e$, and $S = S_e \cup \bigl( (n/2) + m S_e \bigr)$.
\end{cor}
\end{thmref}

It also makes it possible to strengthen \cref{OddCirculant} to a determination of all the possible orders of nontrivially unstable circulant graphs:

\begin{thmref}{orders}
\begin{cor}
For $n \in \ZZ^+$, there does \textbf{not} exist a nontrivially unstable circulant graph of order~$n$ if and only if either $n$~is odd, or $n < 8$, or $n = 2p$, for some prime number $p \equiv 3 \pmod{4}$.
\end{cor}
\end{thmref}

Here is an outline of the paper. 
After this introduction comes a short section of preliminaries.  
Our main results, which were described above, are proved in 
 	\cref{InstabilitySect} (\cref{WilsonC1C3,IsoTranslateS,XeUnstable})
	and 
	\cref{2pSect} (\cref{2p} and its corollaries). 
In between, \cref{mSSect} discusses a lemma from~\cite{Morris-OddAbelian}.
Finally, \cref{ComputationalSect} briefly presents a few minor results that were obtained by computer exploration.

\begin{ack}
The work of Ademir Hujdurovi\'c  is supported in part by the Slovenian Research Agency (research program P1-0404 and research projects N1-0102, N1-0140, N1-0159, N1-0208, J1-1691, J1-1694, J1-1695, J1-2451 and J1-9110).
\end{ack}

\section{Preliminaries}

For emphasis, and ease of reference, we repeat a basic assumption from the first paragraph of the introduction (and add an exception):

\begin{assump} \label{GraphsAreSimple}
All graphs in this paper are finite, undirected, and simple (no loops or multiple edges), except that loops will be allowed in \cref{mSSect} (see \cref{AllowLoops}).
\end{assump}

\begin{notation} \label{nh}
For convenience, proofs will sometimes use the following abbreviation:
	\[ \nh = n/2 .\]
\end{notation}

\begin{defn} \label{CayleyDefn}
Let $S$ be a subset of an abelian group~$G$, such that $-s \in S$ for all $s \in S$, and $0 \notin S$.
	\begin{enumerate}
	\item The \emph{Cayley graph} $\Cay(G, S)$ is the graph whose vertices are the elements of~$G$, and with an edge from $v$ to~$w$ if and only if $w = v + s$ for some $s \in S$ (cf.\ \cite[\S1]{Li-IsoSurvey}).
	\item For $s \in S$, we let $\tilde s = (s,1)$. 
	\item Note that if $X = \Cay(G, S)$, and we let $\widetilde S = \{\, \tilde s \mid s \in S \,\}$, then
		\[ B X = \Cay \bigl( G \times \ZZ_2, \widetilde S \, \bigr) . \]
	For $s \in S$, we say that an edge $uv$ of $BX$ is an \emph{$s$-edge} if $v = u \pm \tilde s$.
	\end{enumerate}
\end{defn}

\begin{rem} \label{no0}
The reason that the set~$S$ in $\Cay(G, S)$ is not allowed to contain~$0$ is that \cref{GraphsAreSimple} forbids the graphs in this paper from having loops.
\end{rem}

\begin{defn}[Kotlov-Lovász {\cite{KotlovLovasz-twinfree}}] \label{TwinFreeDefn} 
A graph~$X$ is \emph{twin-free} if there do not exist two distinct vertices that have exactly the same neighbors.
\end{defn}

\begin{rem} \label{TwinFreeIffNoTranslation}
It is easy to see (and well known) that $\Cay(\ZZ_n, S)$ is twin-free if and only if there does not exist a nonzero $h \in \ZZ_n$, such that $h + S = S$.\refnote{TwinFreeIffNoTranslation-aid}
\end{rem}

The notion of ``block'' is a fundamental concept in the theory of permutation groups, but we need only the following special case:

\begin{defn}[cf.\ {\cite[pp.~12--13]{DixonMortimer}}] \label{BlockDefn}
Let $X = \Cay(\ZZ_n, S)$ be a circulant graph. A nonempty subset~$\mathcal{B}$ of $V(BX)$ is a \emph{block} for the action of $\Aut BX$ if, for every $\alpha \in \Aut BX$, we have 
	\[ \text{either \ $\alpha(\mathcal{B}) = \mathcal{B}$ \ or \ $\alpha(\mathcal{B}) \cap \mathcal{B} = \emptyset$.} \]
It is easy to see that this implies $\mathcal{B}$ is a coset of some subgroup~$H$ of $\ZZ_n \times \ZZ_2$, and that every coset of~$H$ is a block.\refnote{BlockDefn-H} Indeed, the action of $\Aut BX$ permutes these cosets, so there is a natural action of $\Aut BX$ on the set of cosets.
\end{defn}

\begin{defn}[cf.\ {\cite[Defn.~3.1]{Li-IsoSurvey}}] \label{CIDefn}
To say that a circulant graph $X = \Cay(\ZZ_n, S)$ has the \emph{Cayley Isomorphism Property} means that if $S'$ is a subset of~$\ZZ_n$, such that $X \cong \Cay(\ZZ_n, S')$, then there exists $m \in \ZZ_n^\times$, such that $S' = m S$.
\end{defn}

\begin{thm}[Muzychuk \cite{Muzychuk-Adam}] \label{Muzychuk-CI}
If $n$ is either square-free, or twice a square-free number, then every Cayley graph on~$\ZZ_n$ has the Cayley Isomorphism Property.
\end{thm}

\section{Some conditions that imply instability} \label{InstabilitySect}

The following known result is elementary, but, for the reader's convenience, we briefly recall the proof.

\begin{lem}[{\cite[Thm.~3.2]{LauriMizziScapellato}, \cite[Prop.~4.2]{MarusicScapellatoSalvi}}] \label{UnstableIffPerms}
Let $X$ be a graph.
If there exist permutations $\alpha$ and~$\beta$ of $V(X)$, such that $\alpha \neq \beta$ and, for every edge~$uv$ of~$X$, the vertex $\alpha(u)$ is adjacent to $\beta(v)$, then $X$ is unstable.

The converse holds if $BX$ is connected.
\end{lem}

\begin{proof}[Sketch of proof]
($\Rightarrow$) Define $\varphi \in \Aut(BX)$ by 
	\[ \varphi(v,i) = \begin{cases}
		\bigl( \alpha(v), i \bigr) & \text{if $i = 0$}, \\
		\bigl( \beta(v) , i \bigr) & \text{if $i = 1$}
		. \end{cases} \]
Since $\alpha \neq \beta$, we have $\varphi \notin \Aut X \times S_2$.

($\Leftarrow$) Let $\varphi \in \Aut BX$, such that $\varphi \notin \Aut X \times S_2$.
Since $BX$ is connected, we know that its bipartition is unique, so the bipartition sets are blocks for the action of $\Aut BX$. Therefore, after composing by a translation, we may assume $\varphi \bigl( \ZZ_n \times \{i\} \bigr) =  
\ZZ_n \times \{i\}$ for $i = 0,1$. So we may define permutations $\alpha$ and~$\beta$ of $V(X)$ by
	\[ \text{$\varphi(v,0) = \bigl( \alpha(v), 0 \bigr)$
	\ and \ 
	$\varphi(v,1) = \bigl( \beta(v), 1 \bigr)$}
	. \qedhere \]
\end{proof}

\begin{thm}[cf.\ {\cite[Thms.~1, C.1, and C.3]{Wilson}}] \label{WilsonC1C3}
Let $X = \Cay(\ZZ_n, S)$ be a circulant graph.
Choose nontrivial subgroups~$H$ and~$K$ of\/~$\ZZ_n$, such that $|K|$ is even, and let $K_o = K \setminus 2K$. If either 
	\begin{enumerate}
	\item \label{WilsonC1C3-inKo}
	$S + H\subseteq S \cup (K_o + H)$ and $H \cap K_o = \emptyset$,
	or
	\item \label{WilsonC1C3-inS}
	$(S \setminus K_o) + H \subseteq S \cup K_o$ and either\/ $|H| \neq 2$ or\/ $|K|$ is divisible by\/~$4$,
	\end{enumerate}
then $X$ is not stable.
\end{thm}

\begin{proof}[Proof \normalfont (cf.\ proof of {\cite[Thm.~1]{Wilson}})]
Let $h$ be a generator of~$H$. We will define permutations $\alpha$ and~$\beta$ of~$\ZZ_n$, such that \cref{UnstableIffPerms} applies.

\medbreak

\pref{WilsonC1C3-inKo}
Define
	\begin{align*}
	 \alpha(x) = \begin{cases}
		x + h & \text{if $x \in 2K + H$}; \\
		x & \text{otherwise};
		 \end{cases} 
	&&
		\beta(x) = \begin{cases}
		x + h & \text{if $x \in K_o + H$}; \\
		x & \text{otherwise}
		. \end{cases} \end{align*}
Note that $0 \notin K_o + H$ (because $H \cap K_o = \emptyset$), so $\beta(0) = 0$.
Since $\alpha(0) = h$, this implies $\alpha \neq \beta$.

Given an edge $uv$ of~$X$, we wish to show that $\alpha(u)$ is adjacent to~$\beta(v)$.  We may assume that either $u$ is moved by~$\alpha$ or $v$~is moved by~$\beta$. In fact, we may assume that exactly one of the vertices is moved, for otherwise, 
	\[ \beta(v) - \alpha(u) = (v + h) - (u + h) = v - u \in S . \]
This means we may assume that either $u \in 2K + H$ or $v \in K_o + H$, but not both. 
Letting $s = v - u \in S$, this implies $s \notin K_o + H$.

Also, we have 
	\[ \beta(v) - \alpha(u) \in (v + H) - (u + H) = (v - u) + H = s + H , \]
so we may write $\beta(v) - \alpha(u) = s + h'$, for some $h' \in H$.
By the first assumption of~\pref{WilsonC1C3-inKo}, we know $s + h' \in S \cup (K_o + H)$. Since $s \notin K_o + H$, this implies $s + h' \in S$, so $\alpha(u)$ is adjacent to $\beta(v)$.

\medbreak

\pref{WilsonC1C3-inS}
If $h \in 2K$, then  $K_o + H = K_o$, so
	\[ ( S \cap K_o) + H \subseteq K_o + H = K_o . \]
Since, by the first assumption of~\pref{WilsonC1C3-inS}, we also have
	$( S \setminus K_o) + H \subseteq S \cup K_o$,
this implies that \pref{WilsonC1C3-inKo} applies.
Therefore, we may assume 
	\[ h \notin 2K . \]

Define
	\begin{align*}
	 \alpha(x) = \begin{cases}
		x + h & \text{if $x \in 2K$}; \\
		x - h & \text{if $x \in 2K + h$}; \\
		x & \text{otherwise}
		 \end{cases} 
	&&
		\beta(x) = \begin{cases}
		x + h & \text{if $x \in K_o$}; \\
		x - h & \text{if $x \in K_o + h$}; \\
		x & \text{otherwise}
		. \end{cases} \end{align*}
		
We claim that $\alpha \neq \beta$. Note that $\alpha(0) = h$.
Therefore, if $\alpha = \beta$, then we must have $\beta(0) = h$. Since $0 \notin K_o$, this implies that $0 \in K_o + h$ (which means $h \in K_o$) and $-h = h$ (which means $|h| = 2$). Since $|h| = 2$ and $h \in K_o$, we see that $|H| = 2$ and that $|K|$ is not divisible by~$4$.\refnote{WilsonC1C3-inS-aid} This contradicts the second half of assumption~\pref{WilsonC1C3-inS}, so the proof of the claim is complete.

Given an edge $uv$ of~$X$, we wish to show that $\alpha(u)$ is adjacent to~$\beta(v)$. That is, we wish to show $\beta(v) - \alpha(u) \in S$. We have $v = u + s$ for some $s \in S$. We may assume
	\begin{align*}
	\beta(v) - \alpha(u) \neq v - u
	. \end{align*}
In particular, we cannot have both $\alpha(u) = u$ and $\beta(v) = v$. Therefore, 
	\[ \text{either \ $u \in 2K \cup (2K + h)$ \ or \ $v \in K_o \cup (K_o + h)$.} \]

\setcounter{case}{0}

\begin{case} \label{WilsonC1C3Pf-and}
Assume $u \in 2K \cup (2K + h)$ \textbf{and} $v \in K_o \cup (K_o + h)$.
\end{case}
We consider two different possibilities, but both of the arguments are very similar.

\begin{subcase}
Assume $u \in 2K$.
\end{subcase}
Then $\alpha(u) = u + h$. 
Since $\beta(v) - \alpha(u) \neq v - u$, this implies $\beta(v) \neq v + h$, so $v \notin K_o$. By the assumption of \cref{WilsonC1C3Pf-and}, this implies $v \in K_o + h$, so $\beta(v) = v - h$. Hence, $\beta(v) - \alpha(u) = s - 2h$.

We have $u \in 2K$ and $v \in K_o + h$, so $s = v - u \in K_o + h$, which means $s - h \in K_o$. Since $h \notin 2K$, this implies $s \notin K_o$ and $s - 2h \notin K_o$. Since $s \notin K_o$, the first assumption of \pref{WilsonC1C3-inS} tells us $s + H  \subseteq S \cup K_o$. Since $s - 2h \notin K_o$, this implies $s - 2h \in S$. So $\alpha(u)$ is adjacent to $\beta(v)$.

\begin{subcase}
Assume $u \in 2K + h$. 
\end{subcase}
We have $\alpha(u) = u - h$. 
Since $\beta(v) - \alpha(u) \neq v - u$, this implies $\beta(v) \neq v - h$, so $v \notin K_o + h$. By the assumption of \cref{WilsonC1C3Pf-and}, this implies $v \in K_o$, so $\beta(v) = v + h$. Hence, $\beta(v) - \alpha(u) = s + 2h$.

We have $u \in 2K + h$ and $v \in K_o$, so $s = v - u \in K_o - h$, which means $s + h \in K_o$. Since $h \notin 2K$, this implies $s \notin K_o$ and $s + 2h \notin K_o$. Since $s \notin K_o$, the first assumption of \pref{WilsonC1C3-inS} tells us $s + H  \subseteq S \cup K_o$. Since $s + 2h \notin K_o$, this implies $s + 2h \in S$. So $\alpha(u)$ is adjacent to $\beta(v)$.

\begin{case}
Assume \cref{WilsonC1C3Pf-and} does not apply.
\end{case}
As in \cref{WilsonC1C3Pf-and}, we consider two different possibilities, but both of the arguments are very similar.

\begin{subcase}
Assume $u \in 2K \cup (2K + h)$.
\end{subcase}
Choose $\delta \in \{0,1\}$, such that $u \in 2K + \delta h$.
We have $\alpha(u) = u + \epsilon h$, where $\epsilon = 1 - 2\delta$, and we also have $v \notin K_o \cup (K_o + h)$, since \cref{WilsonC1C3Pf-and} does not apply, so $\beta(v) = v$. Since $u \in 2K + \delta h$, but $u + s = v \notin K_o + \delta h$, we have $s \notin K_o$. So the first assumption of~\pref{WilsonC1C3-inS} tells us $s + H  \subseteq S \cup K_o$, so $s - \epsilon h \in S \cup K_o$. Since $\beta(v) - \alpha(u) = s - \epsilon h$, then we may assume $s - \epsilon h \in K_o$ (otherwise, $\alpha(u)$ is adjacent to $\beta(v)$, as desired), so $s \in K_o + \epsilon h$. 
Then
	\[ v = u + s \in (2K + \delta h) + (K_o + \epsilon h) = K_o + (\delta + \epsilon) h = K_o + (1 - \delta) h. \]
Since $1 - \delta \in \{0,1\}$, but $v \notin K_o \cup (K_o + h)$, this is a contradiction.

\begin{subcase}
Assume $v \in K_o \cup (K_o + h)$.
\end{subcase}
Choose $\delta \in \{0,1\}$, such that $v \in K_o + \delta h$.
We have $\beta(v) = v + \epsilon h$, where $\epsilon = 1 - 2\delta$, and we also have $u \notin 2K \cup (2K + h)$, since \cref{WilsonC1C3Pf-and} does not apply, so $\alpha(u) = u$. Since $v \in K_o + \delta h$, but $v - s = u \notin 2K + \delta h$, we have $s \notin K_o$. So the first assumption of~\pref{WilsonC1C3-inS} tells us $s + H  \subseteq S \cup K_o$, so $s + \epsilon h \in S \cup K_o$. Since $\beta(v) - \alpha(u) = s + \epsilon h$, then we may assume $s + \epsilon h \in K_o$ (otherwise, $\alpha(u)$ is adjacent to $\beta(v)$, as desired), so $s \in K_o - \epsilon h$.
Then
	\[ u = v - s \in (K_o + \delta h) - (K_o - \epsilon h) = 2K + (\delta + \epsilon) h = 2K + (1 - \delta) h. \]
Since $1 - \delta \in \{0,1\}$, but $u \notin 2K \cup (2K + h)$, this is a contradiction.
\end{proof}

\begin{rem} \label{RestateWilsonC1C3-inKo}
In the notation of \cref{WilsonC1C3}, it is clear that 
	\[ \bigl(S \cap (K_o + H) \bigr) + H \subseteq (K_o + H) + H = K_o + H . \]
Therefore, the first condition of~\fullref{WilsonC1C3}{inKo} can be restated as: 
	\[ \bigl( S \setminus (K_o + H) \bigr) + H \subseteq S . \]
\end{rem}

Wilson types~\pref{Wilson-C1}, \pref{Wilson-C2}, and~\pref{Wilson-C3} are special cases of \fullcref{WilsonC1C3}{inS}:

\begin{prop} \label{WilsonIsSpecialCase}
If $\Cay(\ZZ_n, S)$ has Wilson type~\pref{Wilson-C1}, \pref{Wilson-C2}, or~\pref{Wilson-C3}, then there are nontrivial subgroups $H$ and~$K$ of\/~$\ZZ_n$ that satisfy the conditions given in part~\pref{WilsonC1C3-inS} of \cref{WilsonC1C3} \textup(and $|K|$ is even\textup).
\end{prop}

\begin{proof}
\pref{Wilson-C1}
Let $K = \ZZ_n$ and $H = \langle h \rangle$. Then
	\[ (S \setminus K_o) + H
	= S_e + \langle h \rangle
	= S_e
	 \subseteq S , \]
so the first condition of~\fullref{WilsonC1C3}{inS} is satisfied.
Also, since $h \in 2\ZZ_n = 2K$, it must be true that either $|H| \neq 2$ or $|K|$ is divisible by\/~$4$.\refnote{WilsonIsSpecialCase-C1-aid}

\pref{Wilson-C2} 
Let $H = \langle h \rangle$ and $K = 2\ZZ_n$. (Note that $|H| > 2$, since $h$ is odd and $n$~is divisible by~$4$.) Then 
	\[ K_o = \{\, x \in \ZZ_n \mid x \equiv 2 \pmod{4} \,\} . \]
We will show that $(S \setminus K_o) + H \subseteq S \cup K_o$.

We may assume $h \equiv 1 \pmod{4}$, by applying the graph automorphism $x \mapsto -x$ if necessary.
Fix some $s \in S \setminus K_o$. 

Suppose, first, that $s \not\equiv 0 \pmod{4}$ (and recall that $s \notin K_o$, so $s \not\equiv 2 \pmod{4}$), so $s$ is odd. This means $s \in S_o$, so we see from~\fullref{Wilson-C2}{So} that $s + 2kh \in S$ for all $k \in \ZZ$. If $s + 2kh \equiv 3 \pmod{4}$, then $s + (2k + 1)h \in S$ (by~\fullref{Wilson-C2}{s+b}). If $s + 2kh \equiv 1 \pmod{4}$, then $s + (2k + 1)h \in K_o$. Thus, we have $s + H \subseteq S \cup K_o$.

Now, suppose $s \equiv 0 \pmod{4}$. Then $s + h \in S$ (by~\fullref{Wilson-C2}{s+b}). Now, since $s + h \not\equiv 0 \pmod{4}$, the previous case tells us that $s + h + H \subseteq S \cup K_o$. Since $h + H = H$, this means $s + H \subseteq S \cup K_o$.

\pref{Wilson-C3} 
Let $K = \langle R \rangle = \langle d \rangle$. 
Since $n/d$ is even, we know that $K$ has even order. Then, since $r/d$~is odd for every $r \in R$, we see that $R \subseteq K_o$. 
By the definition of~$R$, this means $(S \setminus K_o) + H \subseteq S$, so the first condition of~\fullref{WilsonC1C3}{inS} is satisfied.

Also, since either $H \nsubseteq d\ZZ_n$ or $H \subseteq 2d\ZZ_n$, we know that either $H \nsubseteq K$ or $H \subseteq 2K$. If $H \nsubseteq K$ and $|H| = 2$, then it is clear that $H \cap K = \{0\} \subseteq 2K$. Thus, in both cases, we have $H \cap K \subseteq 2K$, which easily implies that either $|H| \neq 2$ or $|K|$ is divisible by~$4$.\refnote{WilsonIsSpecialCase-C3-aid}
\end{proof}

There is a strong converse to \cref{WilsonIsSpecialCase} when $n$ is not divisible by~$4$:

\begin{prop} \label{NotDivBy4=C1}
If $X = \Cay(\ZZ_n, S)$, $H$, and~$K$ satisfy the conditions of \fullcref{WilsonC1C3}{inS}, and $n$ is not divisible by~$4$, then $X$ has Wilson type~\pref{Wilson-C1}.
\end{prop}

\begin{proof}
Since $n$ is not divisible by~$4$, it is not possible for $|K|$ to be divisible by~$4$, so the second half of \fullcref{WilsonC1C3}{inS} tells us that $|H| > 2$. This implies that $H_e \coloneqq H \cap 2\ZZ_n$ is nontrivial. 
Also, since $n$ is not divisible by~$4$, we know that $K_o \cap 2\ZZ_n = \emptyset$, so
	\[ S_e + H_e
	= (S  \cap 2 \ZZ_n) + H_e
	\subseteq (S \setminus K_o) + H
	\subseteq S \cup K_o
	\subseteq S \cup (\ZZ_n \setminus 2 \ZZ_n) . \]
Since $H_e \subseteq 2\ZZ_n$, we also know that $S_e + H_e \subseteq 2 \ZZ_n$. Therefore, we conclude that $S_e + H_e \subseteq S_e$, so $X$ has Wilson type~\pref{Wilson-C1}.
\end{proof}

\begin{rem} \label{OnlyC1C4}
By combining \cref{WilsonIsSpecialCase,NotDivBy4=C1}, we see that if $X$ has a Wilson type, and $n$ is not divisible by~$4$, then $X$ must have Wilson type~\pref{Wilson-C1} or~\pref{Wilson-C4}.
\end{rem}

\begin{prop} \label{IsoTranslateS}
Assume $X = \Cay(\ZZ_n, S)$ is a circulant graph of even order. If $X \cong \Cay \bigl( \ZZ_n, S + (n/2) \bigr)$, then $X$ is unstable.
\end{prop}

\begin{proof}
If $\alpha$ is an isomorphism from $\Cay(\ZZ_n, S)$ to $\Cay(\ZZ_n, S + \nh)$, then \cref{UnstableIffPerms} applies with $\beta(x) = \alpha(x) + \nh$.\refnote{IsoTranslateS-aid}
\end{proof}

\begin{rems}
\leavevmode
	\begin{enumerate}
	\item \Cref{IsoTranslateS} is a generalization of Wilson type~\pref{Wilson-C4}. 
	To see this, note that if $\nh + mS = S$, then $mS = S + \nh$. Also, it is well known that if $m \in \ZZ_n^\times$, then $\Cay(\ZZ_n, S) \cong \Cay(\ZZ_n, mS)$ (cf.\ \cite[Lem.~3.7.3, p.~48]{GodsilRoyle}). Therefore, we conclude that $\Cay(\ZZ_n, S) \cong \Cay(\ZZ_n, S + \nh)$, so \cref{IsoTranslateS} applies.
	\item M.\,Muzychuk \cite{Muzychuk-IsoCirculant} found an efficient method to check whether two circulant graphs (such as $\Cay(\ZZ_n, S)$ and $\Cay(\ZZ_n, S + \nh)$) are isomorphic.
	\end{enumerate}
\end{rems}

Here is a family of examples that illustrate \cref{IsoTranslateS}:

\begin{eg} \label{IsoTranslateSEg}
Let $n = 2 p^2$, where $p$~is prime and $p \equiv 1 \pmod{4}$, and choose $c \in \ZZ$, such that $c^2 \equiv -1 \pmod{p}$. Fix some $a \in \ZZ_n$ of order~$p$, and let $S = S_e \cup S_o$, where
	\begin{align*}
	S_e &= \bigl(\pm 2 + \langle a \rangle\bigr) \cup \{\pm a\} \subseteq 2 \ZZ_n ,  \\
	S_o' &= \bigl( \pm 2 + \langle a \rangle\bigr) \cup \{\pm c a \}\subseteq 2 \ZZ_n, \\
	S_o &= \nh + S_o' \subseteq 1 + 2 \ZZ_n
	. \end{align*}
Then 
	\begin{enumerate}
	\item \label{IsoTranslateSEg-iso}
	$\Cay(\ZZ_n, S) \cong \Cay \bigl( \ZZ_n, S + (n/2) \bigr)$, so \cref{IsoTranslateS} implies that $\Cay(\ZZ_n, S)$ is (nontrivially) unstable, but
	\item \label{IsoTranslateSEg-notWilson}
	$\Cay(\ZZ_n, S)$ does not have a Wilson type.
	\end{enumerate}
\end{eg}

\begin{proof}
\pref{IsoTranslateSEg-iso}
Choose a set~$\mathcal{R}$ of coset representatives for $\langle a \rangle$ in~$\ZZ_n$, such that $\mathcal{R} + \nh = \mathcal{R}$, and define $\alpha \colon \ZZ_n \to \ZZ_n$ by 
	\[ \text{$\alpha(r + x) = r + cx$ \ for $r \in \mathcal{R}$ and $x \in \langle a \rangle$} .\]
It suffices to show that if $v,w \in \ZZ_n$, such that $v - w \in S$, then $\alpha(v) - \alpha(w) \in S + \nh$.

First, consider two vertices $v = r + x$ and $w = r + y$ that are in the same coset of $\langle a \rangle$. Then the definition of~$S$ implies that $v - w = \pm a$, so $y = x \pm a$, so 
	\[ \alpha(v) - \alpha(w) 
	= \bigl( r + cx \bigr) -  \bigl( r + c (x \pm a) \bigr)
	= \pm ca \in S_o'
	= S_o + \nh . \]

Next, suppose $v \in w + \nh + \langle a \rangle$. Assume, without loss of generality, that $v \in 1 + 2\ZZ_n$ and $w \in 2\ZZ_n$. Write $v = r + \nh + x$ and $w = r + y$ with $r \in \mathcal{R}$ and $x,y \in \langle a \rangle$. The definition of~$S$ implies that $v - w = \nh \pm ca$, so $y = x \pm ca$, so (using the fact that $c^2 \equiv -1 \pmod{p}$) we have
 	\[ \alpha(v) - \alpha(w) 
	= \bigl( r + \nh + cx \bigr) -  \bigl( r + c (x \pm ca) \bigr)
	= \nh \pm c^2a
	= \nh \mp a
	\in \nh + S_e
	 . \]

We may now assume that $v$ and~$w$ are in two different cosets of $\langle a, \nh \rangle$. Then, from the definition of~$S$, we see that every vertex in $v + \langle a \rangle$ is adjacent to every vertex in $w + \langle a \rangle$. Since $\alpha(v) \in v + \langle a \rangle$ and $\alpha(w) \in w + \langle a \rangle$, it is therefore obvious that $\alpha(v)$ is adjacent to $\alpha(w)$.

\pref{IsoTranslateSEg-notWilson}
The proof is by contradiction.

Suppose, first, that $\Cay(\ZZ_n, S)$ has Wilson type \pref{Wilson-C1}, \pref{Wilson-C2}, or~\pref{Wilson-C3}. Then \cref{OnlyC1C4} tells us that the graph actually has Wilson type~\pref{Wilson-C1}, so $h+S_e=S_e$ for some non-zero $h \in 2\ZZ_n$.
Since $2\ZZ_n \cong \ZZ_{p^2}$, we know that $|h|$ is divisible by~$p$. Since $S_e$ is a union of cosets of~$\langle h \rangle$, this implies that $|S_e|$ is divisible by~$p$, which contradicts the fact that $|S_e| = 2|a| + 2 = 2p + 2$.

We may now assume that the graph is of Wilson type \pref{Wilson-C4}. Then we can find $m\in \ZZ_{2p^2}^\times$, such that $mS+\nh=S$. Since $\nh$ is odd, this implies $m S_e + \nh = S_o$, so $m S_e = S_o'$. By passing to the quotient group $2\ZZ_n / \langle a \rangle$, we conclude that $m \equiv \pm 1 \pmod{p}$.\refnote{IsoTranslateSEg-notWilson-aid} So $ma = a \notin S_o'$. This contradicts the fact that $m S_e = S_o'$.
\end{proof}

It is shown in \cite{HujdurovicMitrovicMorris} that every nontrivially unstable circulant graph of valency $\le 7$ has a Wilson type, so the following examples have minimal valency among those that do not have a Wilson type:

\begin{eg} \label{val8Eg}
Let $n \coloneqq 3 \cdot 2^\ell$, where $\ell \ge 4$ is even, and let
	\[ S \coloneqq \left\{ \pm 3, \pm 6, \pm \frac{n}{12}, \frac{n}{2} \pm 3 \right\} . \]
Then the circulant graph $X \coloneqq \Cay(\ZZ_n, S)$ has valency~$8$ and is nontrivially unstable, but does not have a Wilson type.
\end{eg}

\begin{proof}
It is easy to see that $X$ is connected, nonbipartite, and twin-free.\refnote{val8Eg-aid}
For convenience, let $a = n/12$.

(unstable) 
Let $\rho_0 \colon 2\ZZ_n \to \ZZ_2$ be the homomorphism with kernel $4\ZZ_n$. Then define $\rho \colon \ZZ_n \to \ZZ_2$ by
	\[ \rho(v) = \begin{cases}
	\rho_0(v) & \text{if $v \in 2\ZZ_n$} ; \\
	\rho_0(v + 1) & \text{otherwise} 
	. \end{cases} \]
Finally, we let $m \coloneqq (n/6) - 1$, and define 
	\[ \alpha(v) = mv + \rho(v) \nh . \]
We will show that $\alpha$ is an isomorphism from~$X$ to $X \cong \Cay(\ZZ_n, S + \nh)$, so \cref{IsoTranslateS} shows that $X$ is unstable.
Thus, given $v \in \ZZ_n$ and $s \in S$, we wish to show that 
	\begin{align} \label{val8EgPf-inS}
	\alpha(v + s) - \alpha(v) \in S + \nh
	. \end{align}

From the choice of~$m$, we have $m \cdot 3 = \nh - 3$, and a straightforward calculation shows that $m(\nh + 3) = -3$ (in~$\ZZ_n$),\refnote{val8Eg-unstable-aid}
so
	\[ m \{\pm 3, \nh \pm 3\} + \langle \nh \rangle = \{\pm 3, \nh \pm 3\} \subseteq S . \]
Since it is clear from the definition of~$\alpha$ that $\alpha(v + s) - \alpha(v) \in ms + \langle \nh \rangle$, this implies that \pref{val8EgPf-inS} holds for $s \in \{\pm 3, \nh \pm 3\}$.

To deal with the remaining elements $\pm 6$ and~$\pm a$ of~$S$, first note that $6 \notin 4\ZZ_n$ and $a = 2^{\ell-2} \in 4\ZZ_n$ (since $\ell \ge 4$), so $\rho_0(6) = 1$ and $\rho_0(a) = 0$. Hence, for all $v \in \ZZ_n$, we have
	\[ \text{$\rho(v \pm 6) = \rho(v) + 1$ and $\rho(v \pm a) = \rho(v)$.} \]
Therefore, we have
	\begin{align*}
	\alpha(v \pm 6) - \alpha(v) 
	&= \bigl( m(v \pm 6) + \rho(v  \pm 6) \nh \bigr) - \bigl( mv + \rho(v) \nh \bigr)
	\\&= \pm 6m + \bigl( \rho(v \pm 6) - \rho(v) \bigr) \nh
	\\&=  \pm(n - 6) + \nh
	\\&\equiv  \pm 6 + \nh
		&& \pmod{n}
	\\&\in S + \nh
	 .\end{align*}
Also note that $2^{\ell-1} \equiv 8 \pmod{12}$ (because $\ell - 1 \ge 3$ is odd), so we have $m = 2^{\ell-1} - 1 = 6k + 1$, where $k$ is odd. Therefore
	\begin{align*}
	\alpha(v \pm a) - \alpha(v) 
	&= \bigl( m(v \pm a) + \rho(v \pm a) \nh \bigr) - \bigl( mv + \rho(v) \nh \bigr)
	\\&= \pm ma + \bigl( \rho(v \pm a) - \rho(v) \bigr) \nh
	\\&= \pm (6k + 1)a + 0 \nh
	\\&= \pm (k \nh + a)
	\\&= \pm a + \nh
	\\&\in S + \nh
	 . \end{align*}

(no Wilson type)
First, suppose that the graph has Wilson type \pref{Wilson-C1}, \pref{Wilson-C2}, or \pref{Wilson-C3}. By \cref{WilsonIsSpecialCase}, this implies there are subgroups $H$ and~$K$ of~$\ZZ_n$ that satisfy the conditions of \fullcref{WilsonC1C3}{inS}. In particular, at least one coset of~$H$ is completely contained in~$S$. 

We claim that $H = \langle \nh \rangle$. Let $h = n/|H|$, so $h$ is a divisor of~$n$, and $H = \langle h \rangle$. Since $S$ contains a coset of~$H$, we know that at least one of any $h$ consecutive elements $x + 1, x + 2, \ldots, x + h$ of~$\ZZ_n$ is an element of~$S$.
Since $(\nh - 3) - (n/12) > n/3$, this easily implies $h = \nh$,\refnote{val8Eg-noWilson}  which completes the proof of the claim. 

Since $6 + \nh \notin S$ and $(n/12) + \nh \notin S$, we then see from the first half of \fullref{WilsonC1C3}{inS} that $6 \in K_o$ and $n/12 \in K_o$. This is impossible, because $n/12$ is divisible by~$4$, but $6$ is not.

Now suppose that the graph has Wilson type~\pref{Wilson-C4}. This means there is some $m \in \ZZ_n^\times$, such that $mS = S + \nh$. Since $\nh$ is even, this implies $mS_o = S_o + \nh$, so (perhaps after replacing $m$ with $-m$, we have $m \cdot 3 \in \{3, \nh + 3\}$. But then
	\[ m \cdot 6 + \nh = 2(m \cdot 3) + \nh \in 2 \{3, \nh + 3\} + \nh = \{6 + \nh\} . \]
Since $6 + \nh \notin S$, this is a contradiction.
\end{proof}

\begin{prop} \label{XeUnstable}
Assume $X = \Cay(\ZZ_n, S)$ is a circulant graph of even order.
If there exist permutations $\alpha$ and~$\beta$ of~$2 \ZZ_n$, and a subgroup~$H$ of~$2\ZZ_n$, such that:
	\begin{enumerate}
	\item $\alpha \neq \beta$,
	\item \label{XeUnstable-even}
	if the vertices $u,v \in 2 \ZZ_n$ are adjacent, then the vertices $\alpha(u)$ and $\beta(v)$ are also adjacent,
	\item \label{XeUnstable-odd}
	$s + H \subseteq S$, for all odd $s \in S$,
	and
	\item \label{XeUnstable-H}
	$\alpha(v) - v \in H$ and $\beta(v) - v \in H$, for all $v \in 2\ZZ_n$,
	\end{enumerate}
then $X$ is unstable.
\end{prop}

\begin{proof}
Define permutations $\alpha'$ and $\beta'$ of~$\ZZ_n$ by stipulating that 
	\[ \text{$\alpha'(v) = \beta'(v) = v$ if $v$ is odd,} \]
whereas 
	\[ \text{$\alpha'(v) = \alpha(v)$ and $\beta'(v) = \beta(v)$ if $v$ is even.} \]
It is straightforward to verify that if $uv$ is an edge of~$X$, then $\alpha'(u)$ is adjacent to $\beta'(v)$,\refnote{XeUnstable-aid} so \cref{UnstableIffPerms} tells us that $X$ is unstable.
\end{proof}

The following result is the special case where $\Cay(2\ZZ_n, S \cap 2 \ZZ_n)$ has Wilson type~\pref{Wilson-C4}.

\begin{cor} \label{XeC4}
Assume $X = \Cay(\ZZ_n, S)$ is a circulant graph with $n \equiv 0 \pmod{4}$. If there exists $m \in \ZZ_n^*$, such that $m S_e + \nh = S_e$, and $S_o + 2(m-1)\ZZ_n = S_o + \nh = S_o$, then $X$ is unstable.
\end{cor}

\begin{proof}
Define permutations $\alpha$ and~$\beta$ of $2 \ZZ_n$ by $\alpha(x) = mx$ and $\beta(x) = mx + \nh$. (The assumption that $n \equiv 0 \pmod{4}$ implies that $\nh$ is even, so the image of~$\beta$ is in $2 \ZZ_n$.) Let $H = \langle 2(m-1)\ZZ_n, \nh \rangle$, so $S_o + H = S_o$, by assumption. Then \cref{XeUnstable} applies, since $\alpha(v) - v = (m-1)v \in 2(m-1) \ZZ_n \subseteq H$ and $\beta(v) - v = (m-1)v + \nh \subseteq H$.
\end{proof}

\begin{rem} \label{WilsonCorrections}
The statement of \cref{Wilson} includes corrected versions \pref{Wilson-C2} and~\pref{Wilson-C3} of the original statements of Theorems~C.2 and~C.3 that appear in~\cite{Wilson}.
	
\pref{Wilson-C2} Y.-L.\,Qin, B.\,Xia, and S.\,Zhou \cite{QinXiaZhou} corrected the original hypothesis~(C.2) by adding condition~\pref{Wilson-C2-s+b}.
	(Note that complete graphs with more than two vertices are stable \cite[Eg.~2.2]{QinXiaZhou}, but satisfy condition~\pref{Wilson-C2-So} with $h = 1$, so this condition alone does not imply instability, even when $n$ is divisible by~$4$.) 
	
\pref{Wilson-C3} The original statement of hypothesis~(C.3) in~\cite[p.~376]{Wilson} includes the extraneous hypothesis that $d > 1$, and neglects to state the requirements that $n/d$ is even and that either $H \nsubseteq d \ZZ_n$ or $H \subseteq 2d \ZZ_n$. 
	(\emph{Explanation:} \cite[Thm.~1]{Wilson} does not require the Cayley graph generated by the red edges to be disconnected, so there is no need to assume $d > 1$.
	The proof of \cite[Thm.~C.3]{Wilson} uses the assumption that $r/d$ is odd to conclude that each component of $\Cay(\ZZ_n, R)$ is bipartite; this requires each $r \in R$ to be an element of even order in~$\ZZ_n$, which means that $n/{\gcd(r,n)}$ is even.
	Conditions~(1) and~(2) near the bottom of \cite[p.~360]{Wilson} translate to the requirement that either $H \nsubseteq d \ZZ_n$ or $H \subseteq 2d \ZZ_n$.)

It was mentioned above that if $n > 2$, then the complete graph $K_n$ is stable \cite[Eg.~2.2]{QinXiaZhou}. However, if $n$ is even, then $K_n$ satisfies the  conditions in the original statement of~(C.3) (with $H = \langle n/2 \rangle$, $R = \{n/2\}$, $d = n/2$, and $r/d = 1$ for the unique element~$r$ of~$R$). This shows that the additional conditions in~\pref{Wilson-C3} cannot be deleted.
\end{rem}

\section{The main lemma of \texorpdfstring{\cite{Morris-OddAbelian}}{[Morris]}} \label{mSSect}

\begin{assump} \label{AllowLoops}
For the proof of \cref{0orbit}, it will be helpful to temporarily relax our standing assumption that all graphs are simple (see \cref{GraphsAreSimple}). Namely graphs are allowed to have loops (but not multiple edges) in this \lcnamecref{mSSect}.
\end{assump}

The following elementary observation is stated only for automorphisms in~\cite{Morris-OddAbelian}, but the same proof applies to isomorphisms.

\begin{lem}[cf.\ {\cite[Lem.~2.2]{Morris-OddAbelian}}] \label{mSiso}
Let $m \in \ZZ^+$, and let $X_1 = \Cay(G_1, S_1)$ and $X_2 = \Cay(G_2, S_2)$ be Cayley graphs, such that
\noprelistbreak
	\begin{enumerate}
	\item $G_1$ and~$G_2$ are abelian, 
	and 
	\item \label{mSiso-neq}
	for $j = 1,2$, we have $ms \neq mt$ for all $s,t \in S_j$, such that $s \neq t$.
	\end{enumerate}
If $\varphi$ is any isomorphism from~$X_1$ to~$X_2$, then $\varphi$ is also an isomorphism from $\Cay(G_1, mS_1)$ to $\Cay(G_2, mS_2)$, where $m S_j = \{\, ms \mid s \in S_j \,\}$.
\end{lem}

\begin{proof}[Proof \normalfont (cf.\ proof of {\cite[Lem.~2.2]{Morris-OddAbelian}})]
Write $m = p_1 p_2 \cdots p_r$, where each $p_i$ is prime, and let $m_i = p_1 p_2 \cdots p_i$ for $0 \le i \le r$. We will prove by induction on~$i$ that $\varphi$ is an isomorphism from $\Cay(G_1, m_iS_1)$ to $\Cay(G_2, m_iS_2)$
The base case is true by assumption, since $m_0 S_j  = 1 S_j = S_j$.

For $v,w \in G_j$, let $\#(v,w)$ be the number of walks of length~$p_i$ from~$v$ to~$w$ in $\Cay(G_j, m_{i-1} S_j)$. These walks are in one-to-one correspondence with the $p_i$-tuples $(s_1,s_2,\ldots,s_{p_i})$ of elements of~$m_{i-1} S_j$, such that 
	\[ s_1 + s_2 + \cdots + s_{p_i} = w - v . \]
Since $G_j$ is abelian, any cyclic rotation of $(s_1,s_2,\ldots,s_{p_i})$ also corresponds to a walk from~$v$ to~$w$. Therefore, the set of these walks can be partitioned into sets of cardinality~$p_i$, unless $w = p_i s + v$, for some $s \in m_{i-1} S_j$, in which case there is a walk of the form 
	$ v, s + v, 2s + v, \ldots, p_i s + v = w$.
(Also note that $s$ is unique, if it exists, by assumption~\pref{mSiso-neq}.) Hence, we see that 
	\[ \#(v,w) \not\equiv 0 \pmod{p_i} \quad \iff \quad \text{$v$ is adjacent to~$w$ in $\Cay( G_j, p_i m_{i-1} S_j)$} . \]
Since $p_i m_{i-1} = m_i$, the desired conclusion  that $\varphi$ is an isomorphism from $\Cay(G_1, m_i S_1)$ to $\Cay(G_2, m_i S_2)$ now follows from the induction hypothesis that $\varphi$ is an isomorphism from $\Cay(G_1, m_{i-1} S_1)$ to $\Cay(G_2, m_{i-1}S_2)$ (and the observation that isomorphisms preserve the value of the function~$\#$).
\end{proof}

\begin{cor} \label{2S'-BX}
Let $X = \Cay(\ZZ_n, S)$ be a circulant graph of even order, let $\varphi$ be an automorphism of $BX$, and let
    \[ S' = \{\, s' \in S \mid s' + (n/2) \notin S \,\} . \]
Then $\varphi$ is an automorphism of\/ $\Cay(\ZZ_n \times \ZZ_2, 2 S' \times \{0\})$.
\end{cor}

\begin{proof}
For every $s \in S$, we have
	\[ \{\, (t,1) \in G \times \{1\} \mid 2(t,1) = 2(s,1) \,\} = \{(s,1), (s + \nh, 1) \} , \]
since $\nh$ is the unique element of order~$2$ in~$\ZZ_n$. Hence, the set
	\[ \{\, (t,1) \in S \times \{1\} \mid 2(t,1) = 2(s,1) \,\} \]
has cardinality~$1$ (which is odd) if $s + \nh \notin S$, and has cardinality~$2$ (which is even) otherwise. Therefore, the desired conclusion is obtained by applying the proof of \cref{mSiso}\refnote{2S'-BX-aid} with $G_1 = G_2 = \ZZ_n \times \ZZ_2$ and $m = 2$, since 
	\[ 2(S' \times \{1\}) = 2S' \times \{0\} . \qedhere \]
\end{proof}

\begin{cor}[{\cite[Rem.~3.1]{Morris-OddAbelian}, \cite[Thm.~23.9(a), p.~58]{Wielandt}}] \label{mSrelprime}
Let $\varphi$ be an automorphism of a Cayley graph\/ $\Cay(G, S)$, and let $m \in \ZZ^+$. If $G$ is abelian and $\gcd \bigl( m, |G| \bigr) = 1$, then $\varphi$ is an automorphism of\/ $\Cay(G, m S)$.
\end{cor}

\begin{proof}
Apply \cref{mSiso} with $G_1 = G_2 = G$ and $S_1 = S_2 = S$.
To verify hypothesis~\pref{mSiso-neq} of the \lcnamecref{mSiso}, note that the map $x \mapsto mx$ is a bijection on~$G$, since $\gcd \bigl( m, |G| \bigr) = 1$.
\end{proof}

\begin{cor} \label{congruentorbit}
Let $\alpha$ be an automorphism of~$BX$, where $X = \Cay(\ZZ_n , S)$ is a circulant graph, let $s,t \in S$, and let $k,\ell \in \ZZ^+$. Assume
	\begin{enumerate}
	\item $\alpha$ maps some $s$-edge to a $t$-edge, 
	\item $ks \in S$,
	\item \label{congruentorbit-congruent}
	$k \equiv \ell \pmod{\gcd\bigl( |s|, |t| \bigr)}$,
	and
	\item \label{congruentorbit-relprime}
	$\gcd \bigl( k, |s| \bigr) = \gcd \bigl( \ell, |t| \bigr) = 1$.
	\end{enumerate}
Then $\ell t \in S$.
\end{cor}

\begin{proof}
By assumption~\pref{congruentorbit-congruent}, there exists $m \in \ZZ^+$, such that $m \equiv k \pmod{s}$ and $m \equiv \ell \pmod{|t|}$. Then assumption~\pref{congruentorbit-relprime} implies that $m$ is relatively prime to $|s|$ and~$|t|$, so, by Dirichlet's Theorem on primes in arithmetic progressions, we may choose a prime $p > 2n$, such that $pm \equiv 1 \pmod{\lcm\bigl( |s|, |t| \bigr)}$. This implies $pks = s$ and $p\ell t = t$. 

Since $p$ is relatively prime to $|\ZZ_n \times \ZZ_2|$, we know from \cref{mSrelprime} that every automorphism of $BX$ is an automorphism of $\Cay \bigl( \ZZ_n \times \ZZ_2 , p \bigl( S \times \{1\} \bigr) \bigr)$. Since $(s,1) = p(ks,1) \in p \bigl( S \times \{1\} \bigr)$, and $\alpha$ maps some $s$-edge to a $t$-edge, this implies that $t \in pS$. Since $t = p \ell t$, and multiplication by~$p$ is a bijection on $\ZZ_n \times \ZZ_2$, this implies $\ell t \in S$.\refnote{congruentorbit-aid}
\end{proof}

Here are two interesting special cases:

\begin{cor} \label{order2orbit}
Let $\alpha$ be an automorphism of~$BX$, where $X = \Cay(\ZZ_n , S)$ is a circulant graph, and let $s,t \in S$. If $\alpha$ maps some $s$-edge to a $t$-edge, and either $\gcd(|s|, |t| \bigr) = 1$, or $S$ contains every element that generates $\langle s \rangle$ \textup(e.g., if $|s| \in \{1,2,3,4,6\}$\textup), then $S$ contains every element that generates~$\langle t \rangle$.
\end{cor}

\begin{proof}
Let $\ell t$ be a generator of~$\langle t \rangle$, so $\gcd \bigl( \ell, |t| \bigr) = 1$.
It suffices to find $k \in \ZZ^+$, such that 
	\[ \text{$ks \in S$, 
	\ $k \equiv \ell \pmod{\gcd\bigl( |s|, |t| \bigr)}$, 
	\ and \ 
	$\gcd \bigl( k, |s| \bigr) = 1$,} \]
for then \cref{congruentorbit} tells us that $\ell t \in S$.

If $\gcd(|s|, |t| \bigr) = 1$, we may let $k = 1$. 

Since $\gcd \bigl( \ell, |t| \bigr) = 1$, we know that $\ell$ is relatively prime to $\gcd \bigl( |s|, |t| \bigr)$, so there is some $k \in \ZZ^+$, such that 
	\[ \text{$k \equiv \ell \pmod{\gcd\bigl( |s|, |t| \bigr)}$
	\ and \ 
	$\gcd \bigl( k, |s| \bigr) = 1$}. \]
(For example, we could take $k$ to be a large prime.) If $S$ contains every element that generates $\langle s \rangle$, then $ks \in S$.
\end{proof}

Recall that the following \lcnamecref{0orbit}'s assumption that $X$ is loopless will automatically be satisfied in all of the following \lcnamecref{mSSect}s of the paper (see \cref{GraphsAreSimple}).

\begin{cor} \label{0orbit}
Let $\alpha$ be an automorphism of~$BX$, where $X = \Cay(\ZZ_n , S)$ is a connected, nonbipartite, loopless, circulant graph, and let $t \in \ZZ_n$. If $\alpha(0,1) = (t,1)$, then $S$ does not contain any generator of the subgroup $\langle t \rangle$.
\end{cor}

\begin{proof}
Let $S^c = G \setminus S$ be the complement of~$S$, and let $X^c = \Cay(\ZZ_n, S^c)$. Since $BX$ is connected, it is easy to see that every automorphism of~$BX$ is also an automorphism of~$BX^c$.\refnote{0orbit-aid}

Note that $0 \in S^c$, since $X$ is assumed to be loopless. Therefore, we may let $s = 0$ (so $|s| = 1$) to conclude from \cref{order2orbit} that $S^c$ contains every generator of $\langle t \rangle$.
\end{proof}

The following result is stated only for automorphisms in~\cite{Morris-OddAbelian}, but essentially the same proof applies to isomorphisms.

\begin{cor}[cf.\ {\cite[Thm.~1.1]{Morris-OddAbelian}}] \label{OddStableIso}
Assume $X_1 = \Cay(\ZZ_n, S_1)$ and $X_2 = \Cay(\ZZ_n, S_2)$ are twin-free, connected, circulant graphs of odd order. If $\varphi$ is any isomorphism from $BX_1$ to $BX_2$, such that 
	\[ \varphi(\ZZ_n \times \{0\}) = \ZZ_n \times \{0\} , \]
then there is an isomorphism $\alpha \colon X_1 \to X_2$, such that $\varphi(x,i) = \bigl( \alpha(x), i \bigr)$ for all $(x,i) \in BX_1$.
\end{cor}

\begin{proof}
This follows quite easily from \cref{mSiso} (with $m = n + 1$), by the argument in the proof of \cite[Thm.~1.1]{Morris-OddAbelian}.\refnote{OddStableIso-aid-mSiso}
\end{proof}

\section{Unstable circulant graphs of order \texorpdfstring{$2p$}{2p}} \label{2pSect}

\begin{thm} \label{2p}
If $p$ is a prime number, then every nontrivially unstable circulant graph of order~$2p$ has Wilson type~\pref{Wilson-C4}.
\end{thm}

The proof of this \lcnamecref{2p} will use several lemmas that are stated in greater generality than is needed here, because they may be useful for understanding the unstable circulant graphs of other square-free orders.

\begin{lem} \label{HalfIsBlock}
Let $X = \Cay(\ZZ_n, S)$ be a connected, nonbipartite, circulant graph, such that $n \equiv 2 \pmod{4}$.
Then $X$ has Wilson type~\pref{Wilson-C1} if and only if $BX$ has an automorphism~$\alpha$, such that
	\begin{enumerate} \itemsep=\smallskipamount
	\item $\alpha \notin \Aut X \times S_2$,
	and
	\item \label{HalfIsBlock-fix}
	$\alpha$ fixes $2\ZZ_n \times \ZZ_2$ \textup(setwise\textup).
	\end{enumerate}
\end{lem}

\begin{proof}
($\Rightarrow$) This is the easy half of the proof (and does not require the assumption that $n \equiv 2 \pmod{4}$).
If $h + S_e = S_e$, then we may define $\alpha \in \Aut(BX)$ by\refnote{2p-alpha-aid}
	\[ \alpha(x,i) = \begin{cases}
		(x + h, i) & \text{if $x \equiv i \pmod{2}$}, \\
		(x, i) & \text{if $x \not\equiv i \pmod{2}$}
		. \end{cases} \]

($\Leftarrow$)
Since $\alpha \notin \Aut X \times S_2$, there is some $v \in \ZZ_n$, such that $\alpha(v,1) \neq \alpha(v,0) + (0,1)$.
After conjugating $\alpha$ by a translation that moves $(v,0)$ to $(0,0)$, we may assume $v = 0$. This means that $\alpha(0,0) = (0,0)$ and $\alpha(0,1) \neq (0,1)$.

Let $S_e = S \cap 2\ZZ_n$, and let $X_e = \Cay(2\ZZ_n, S_e)$, so the subgraph of~$X$ induced by $2\ZZ_n \times \ZZ_2$ is~$BX_e$. Then assumption~\pref{HalfIsBlock-fix} implies that $\alpha$ restricts to an automorphism of~$BX_e$. 

Now, let $X_e'$ be the connected component of~$X_e$ that contains~$0$. Since $n \equiv 2 \pmod{4}$, we know that $n/2$ is odd, so $|\langle S_e \rangle|$ is odd. This implies that $BX_e'$ is connected, and is therefore a connected component of~$BX_e$. Since $BX_e'$ contains the fixed point $(0,0)$, we conclude that $BX_e'$ is $\alpha$-invariant. Therefore, $\alpha$ restricts to an automorphism of $BX_e'$. Since $\alpha$ fixes~$(0,0)$, then it follows from \cref{OddCirculant} that $\alpha(0,1) = (0,1)$. This is a contradiction.
 \end{proof}
 
 \begin{lem}[Klin-Muzychuk, 1995, personal communication] \label{IsoToCirculant}
 Let 
 \noprelistbreak
 	\begin{itemize}
	\item $X = \Cay(\ZZ_n, S)$ be a circulant graph of order~$n$, 
 	\item $G$ be an abelian group of order~$n$, 
 	\item $e \in \ZZ$, such that $eg = 0$ for all $g \in G$,
 	and
 	\item $m \in \ZZ$, such that $m \equiv 1 \pmod{e}$ and $\gcd(m,n) = 1$.
	\end{itemize}
 If $X$ is isomorphic to some Cayley graph $\Cay(G, T)$ on~$G$, then $S = mS$.
 \end{lem}
 
 \begin{proof}[Proof \normalfont(Klin-Muzychuk)]
 Let $\zeta$ be a primitive $n$th root of unity, and, for $1 \le i \le n$, let $\lambda_i = \sum_{s \in S} \zeta^{is}$, so $\lambda_1,\lambda_2, \ldots,\lambda_n$ are the eigenvalues of~$X$ \cite[Prop.~3.5, p.~16]{Biggs}.
 Since $\gcd(m,n) = 1$, there is a Galois automorphism~$\alpha$ of the field extension $\QQ[\zeta]/\QQ$, such that $\alpha(\zeta) = \zeta^m$ \cite[Thm.~6.3.1, p.~278]{Lang-Algebra}. 
  Since $eg = 0$ for all $g \in G$ (and $G$ is abelian), we know that each eigenvalue of~$\Cay(G,T)$ is a sum of $e$th roots of unity. (The range of any homomorphism $\chi \colon G \to \CC^\times$ consists of $e$th roots of unity, so this is a consequence of the well-known formula in \cite[Cor.~3.2]{Babai-spectra} that generalizes the above formula for the eigenvalues of a circulant graph.) Since $m \equiv 1 \pmod{e}$, this implies that every eigenvalue of~$\Cay(G,T)$ is fixed by~$\alpha$. So $\lambda_1,\lambda_2, \ldots,\lambda_n$ are fixed by~$\alpha$. For $1 \le i \le n$, this means 
  	\[ \sum_{s \in S} \zeta^{is} = \lambda_i = \alpha(\lambda_i) = \alpha \left( \sum_{s \in S} \zeta^{is} \right) = \sum_{s \in S} \zeta^{mis} = \sum_{s \in mS} \zeta^{is} . \] 
Since the Vandermonde matrix $[\zeta^{ij}]_{1 \le i,j \le n}$ is invertible \cite{Vandermonde} (and therefore has linearly independent rows), this implies that $S = mS$.\refnote{IsoToCirculant-aid}
 \end{proof}
 
\begin{lem} \label{Interchange}
Assume $X = \Cay(\ZZ_n, S)$ is a nontrivially unstable, circulant graph of even order, and there exists $\alpha \in \Aut BX$, such that $\alpha$ fixes the two cosets of $2\ZZ_n \times \{0\}$ that are in $\ZZ_n \times \{0\}$, but interchanges the two cosets that are in $\ZZ_n \times \{1\}$.
If $X$ does not have Wilson type~\pref{Wilson-C1}, then:
	\begin{enumerate}
	\item \label{Interchange-2mod4}
	$n \equiv 2 \pmod{4}$,
	\item \label{Interchange-generate}
	$\langle S \cap 2\ZZ_n \rangle = 2\ZZ_n$,
	and
	\item \label{Interchange-iso}
	$X \cong \Cay(\ZZ_n, S + (n/2))$ \textup(so \cref{IsoTranslateS} applies\textup).
	\end{enumerate}
\end{lem}

\begin{proof}
Let 
	\begin{itemize}
	\item $S_e = S \cap 2\ZZ_n$, 
	\item $S_o = S \setminus S_e$, 
	\item $G_e = 2\ZZ_n \times \ZZ_2$,
	\item $G_o = \langle 2\ZZ_n \times \{0\} , (1,1) \rangle 
		= \langle (1,1) \rangle
		= \bigl( 2\ZZ_n \times \{0\} \bigr) \cup \bigl( (2 \ZZ_n + 1) \times \{1\} \bigr)$,
	\item $X_o = \Cay(2 \ZZ_n , S_o + \nh)$,
	\item $B_e = BX_e = \Cay \bigl( G_e , S_e \times \{1\} \bigr)$ be the subgraph of~$BX$ induced by~$G_e$, 
	\item $B_o = \Cay \bigl( G_o, S_o \times \{1\} \bigr)$ be the subgraph of~$BX$ induced by~$G_o$,
	and
	\item $\nh = n/2$ (see \cref{nh}).
	\end{itemize}
By assumption, 
	\[ \text{the restriction of~$\alpha$ to~$G_e$ is an isomorphism from~$B_e$ to~$B_o$.} \]

\pref{Interchange-2mod4}
[The proof of this part of the \lcnamecref{Interchange} does not require the assumption that $X$ does not have Wilson type~\pref{Wilson-C1}.]
Suppose $n \not\equiv 2 \pmod{4}$. Then $\gcd(1 + \nh, n) = 1$. Therefore, since $G_o$ is cyclic, and $\nh g = 0$ for all $g \in G_e$,  we see from \cref{IsoToCirculant} that 
	\[ S_o = (1 + \nh) S_o = S_o + \nh . \]
This means that $B_o$ has twins. More precisely, since $|\nh| = 2$, each equivalence class of vertices with the same neighbors has even cardinality. So the same must be true in the isomorphic graph~$B_e$, which means there exists an element $(h_1,h_2)$ of order~$2$ in~$G_e$, such that 
	\[ \bigl( S_e \times \{1\}\bigr) + (h_1,h_2) = S_e \times \{1\} . \]
Since the element of order~$2$ in the cyclic group $2\ZZ_n$ is unique (and the equation $1 + h_2 = 1$ implies that $h_2 = 0$), we conclude that $S_e$ is invariant under translation by~$\nh$. We already know that the same is true for~$S_o$, so we conclude that $S + \nh = S$. This contradicts the assumption that $X$ is nontrivially unstable (and therefore twin-free).

(\ref{Interchange-generate})
Since $n \equiv 2 \pmod{4}$, we know that $\nh$ is odd. So $\ZZ_n = 2\ZZ_n \cup (2\ZZ_n + \nh)$.  For convenience, label the $4$~cosets of $2\ZZ_n \times \{0\}$ by
	\[ \text{$C_e^i = 2\ZZ_n \times \{i\}$
	\  and \ 
	$C_o^i = (2\ZZ_n + \nh) \times \{i\}$ \quad for $i\in \ZZ_2$}
	 .\]
Then $G_e = C_e^0 \cup C_e^1$ and $G_o = C_e^0 \cup C_o^1$.
By assumption, $\alpha$ fixes $C_e^0$ and~$C_o^0$, but interchanges $C_e^1$ and~$C_o^1$.

The function $\varphi(x,i) = (x + i \nh, i)$ is an isomorphism\refnote{Interchange-generate-aid}
	\[ \text{from $B_o = \Cay \bigl( G_o , S_o \times \{1\} \bigr)$ to $BX_o = \Cay \bigl( 2 \ZZ_n \times \ZZ_2,  (S_o + \nh) \times \{1\} \bigr)$.} \]
By composing $\varphi$ with the restriction of~$\alpha$ to~$G_e$ (that is, by restricting the map $\varphi \bigl( \alpha(x,i) \bigr)$ to~$G_e$), we obtain an isomorphism from~$BX_e$ to~$BX_o$. Since $X$ does not have Wilson type~\pref{Wilson-C1}, we know that $X_e$ is twin-free (see \cref{TwinFreeIffNoTranslation}), so we see from \cref{OddCirculant} that each connected component of~$X_e$ is isomorphic to a connected component of~$X_o$.
Hence, these two connected components must have the same order, which means that the two subgroups $\langle S_e \rangle$ and $\langle S_o + \nh \rangle$ of~$\ZZ_n$ have the same order, and are therefore equal. So 
	\[ \langle S_e, \nh \rangle \supseteq \langle S_e \cup S_o \rangle = \langle S \rangle = \ZZ_n . \]
Hence, $\langle S_e \rangle$ is a subgroup of index~$\le 2$ in~$\ZZ_n$, and must therefore be all of~$2\ZZ_n$. This establishes~\pref{Interchange-generate}.

(\ref{Interchange-iso})
Let us begin by making a part of the above proof of~\pref{Interchange-generate} more concrete.
It was established there that composing $\varphi$ with the restriction of~$\alpha$ to~$G_e$ is an isomorphism from~$BX_e$ to~$BX_o$. Since $X_e$ is twin-free, we therefore see from \cref{OddStableIso} that there is an isomorphism $\alpha_0 \colon X_e \to X_o$, such that\refnote{Interchange-iso-aid-XetoXo}
	\[ \text{$\alpha(x,1) = \bigl( \alpha_0(x) + \nh , 1 \bigr)$ for $x \in 2\ZZ_n$} . \]
Similarly, if we let $\varphi'(x,i) = \bigl( x + (1 -  i)\nh, i \bigr)$, then restricting the map $\alpha \bigl( \varphi'(x,i) \bigr) - (\nh, 0)$ to $G_e$ yields an isomorphism from~$BX_o$ to~$BX_e$.\refnote{Interchange-iso-aid-XotoXe} So there is an isomorphism $\alpha_1 \colon X_o \to X_e$, such that
	\[ \text{$\alpha(x,1) = \bigl( \alpha_1(x) + \nh, 1 \bigr)$ for $x \in 2\ZZ_n$} . \]
By comparing the two formulas for $\alpha(x,1)$, we see that $\alpha_0 = \alpha_1$.

Now, define $\alpha' \colon \ZZ_n \to \ZZ_n$ by
	\[ \alpha'(x) = \begin{cases}
		\alpha_0(x) & \text{if $x \in 2\ZZ_n$}, \\
		\alpha_0(x + \nh) + \nh & \text{if $x \notin 2\ZZ_n$}
		. \end{cases} \]
Since $S = S_e \cup S_o $ (and therefore $S + \nh = (S_o + \nh) \cup (S_e + \nh)$), it will suffice to show that 
	\[ \text{$\alpha'$ is an isomorphism from $\Cay(\ZZ_n, S_e)$ to $\Cay(\ZZ_n, S_o + \nh)$} \]
and 
	\[ \text{$\alpha'$ is an isomorphism from $\Cay(\ZZ_n, S_o)$ to $\Cay(\ZZ_n, S_e + \nh)$} .\]

The first is easy to see from the definition of~$\alpha'$, because\refnote{Interchange-iso-alphaprime}
	\[ \text{$\alpha_0$ is an isomorphism from $X_e = \Cay(2 \ZZ_n, S_e)$ to $X_o = \Cay(2\ZZ_n, S_o + \nh)$.} \]

To establish the second, let $s \in S_o$. If $x \in 2 \ZZ_n$, then
	\begin{align*}
	\alpha'(x + s) 
	&= \alpha_0 \bigl( (x + s) + \nh \bigr) + \nh
		&& \text{(definition of~$\alpha'$, since $x + s \notin 2 \ZZ_n$)}
	\\&\in \alpha_0(x) + S_e + \nh
		&& \begin{pmatrix} \text{$\alpha_0 = \alpha_1$ is an isomorphism}
			\\ \text{from $X_o$ to $X_e$} \end{pmatrix}
	\\&= \alpha'(x) + (S_e + \nh)
	.
		&& \text{(definition of~$\alpha'$, since $x \in 2 \ZZ_n$)}
	\end{align*}
Similarly, if $x \notin 2\ZZ_n$, then
	\begin{align*}
		\alpha'(x + s)
		&= \alpha_0(x + s) 
		&& \begin{pmatrix} \text{definition of~$\alpha'$,} \\ \text{since $x + s \in 2 \ZZ_n$} \end{pmatrix}
		\\&= \alpha_0 \bigl( (x + \nh) + (s + \nh) \bigr)
			&& \text{($\nh + \nh = 0$)}
		\\&\in \alpha_0(x + \nh) + S_e
		&& \begin{pmatrix} \text{$\alpha_0 = \alpha_1$ is an isomorphism}
			\\ \text{from $X_o$ to $X_e$} \end{pmatrix}
		\\&= \bigl( \alpha_0(x + \nh) + \nh \bigr) + \bigl( S_e + \nh \bigr)
			&& \text{($\nh + \nh = 0$)}
		\\&= \alpha'(x) + \bigl( S_e + \nh \bigr)
		.
		&& \begin{pmatrix} \text{definition of~$\alpha'$,} \\ \text{since $x \notin 2 \ZZ_n$} \end{pmatrix}
	\end{align*}
Since $\alpha'(x + s) \in \alpha'(x) + (S_e + \nh)$ in both cases, we conclude that $\alpha'$ is an isomorphism from $\Cay(\ZZ_n, S_o)$ to $\Cay(\ZZ_n, S_e + \nh)$, as desired.
This completes the proof of~\pref{Interchange-iso}.
\end{proof}

The following result gathers the most important conclusions of \cref{HalfIsBlock,Interchange}.

\begin{cor} \label{2ZBlock}
Let $X = \Cay(\ZZ_n, S)$ be a nontrivially unstable, circulant graph, such that $n \equiv 2 \pmod{4}$, and such that\/ $2\ZZ_n \times \{0\}$ is a block for the action of $\Aut BX$. Then either $X$~has Wilson type~\pref{Wilson-C1}, or $X$ is isomorphic to $\Cay \bigl( \ZZ_n, S + (n/2) \bigr)$ \textup(so \cref{IsoTranslateS} applies\textup).
\end{cor}

\begin{proof}
Since $X$ is unstable, there exists $\alpha \in \Aut BX$, such that $\alpha \notin \Aut X \times S_2$. By composing with a translation, we may assume that $\alpha$ fixes $(0,0)$. Since $2\ZZ_n \times \{0\}$ is a block for the action of $\Aut BX$, this implies that $\alpha$ fixes $2\ZZ_n \times \{0\}$. It must also fix the bipartition set $\ZZ_n \times \{0\}$, so it fixes the difference
	\[ \bigl( \ZZ_n \times \{0\} \bigr) \setminus \bigl( 2\ZZ_n \times \{0\} \bigr) = (2\ZZ_n + 1) \times \{0\} . \]
And then $\alpha$ either fixes the two remaining cosets, or interchanges them. In the first case, \cref{HalfIsBlock} tells us that $X$ has Wilson type~\pref{Wilson-C1}. In the second case, \fullcref{Interchange}{iso} provides the desired conclusion.
\end{proof}

The following result presents some useful special cases.
Recall that the ``Cayley Isomorphism Property'' was defined in \cref{CIDefn}.

\begin{cor} \label{InterchangeCI}
Assume $X$ is as in \cref{2ZBlock}. Also assume that either
	\begin{enumerate}
	\item \label{InterchangeCI-X}
	$X$ has the Cayley Isomorphism Property, 
	or
	\item \label{InterchangeCI-Xe}
	$X_e$ has the Cayley Isomorphism Property, 
	or
	\item \label{InterchangeCI-squarefree}
	$n$~is square-free,
	or 
	\item \label{InterchangeCI-valency}
	the valency of~$X_e$ is $\le 5$.
	\end{enumerate}
Then $X$ has Wilson type~\pref{Wilson-C1} or~\pref{Wilson-C4}.
\end{cor}

\begin{proof}
Assume $X$ does not have Wilson type~\pref{Wilson-C1}.
Then \cref{2ZBlock} tells us that $X \cong \Cay( \ZZ_n , S + \nh)$.

\pref{InterchangeCI-X}
If $X$ has the Cayley Isomorphism Property, this implies there is some $m \in \ZZ_n^\times$, such that $S + \nh = mS$, so $X$ has Wilson type~\pref{Wilson-C4}.

\pref{InterchangeCI-Xe}
Assume that $X_e$ has the Cayley Isomorphism Property. Let $\alpha_0$ and~$\alpha_1$ be the isomorphisms in the proof of \fullcref{Interchange}{iso}. Since $X_e$ has the Cayley Isomorphism Property, we have $\alpha_0(x) = m \, \varphi(x)$, for some $m \in \ZZ_n^\times$, and some $\varphi \in \Aut X_e$. Then 
	\begin{align*}
	 m S_e 
	 &= m \, \bigl( \varphi(S_e) - \varphi(0) \bigr)
	 	&& \text{($\varphi \in \Aut X_e$)}
	 \\&= \alpha_0(S_e) - \alpha_0(0) 
	 	&& \text{($\alpha_0(x) = m \, \varphi(x)$)}
	 \\&= S_o + \nh 
	 	&& \text{($\alpha_0 \colon X_e \stackrel{\cong}{\to} X_o$)}
	 . \end{align*}
Now, since $\varphi \in \Aut X_e = \Aut \Cay (2\ZZ_n, S_e)$ and $m \in \ZZ_n^\times$, \cref{mSrelprime} implies 
	\[ \varphi \in \Aut \Cay( 2\ZZ_n, m S_e) = \Aut \Cay( 2\ZZ_n, S_o + \nh) = \Aut X_o . \]
Therefore
	\begin{align*}
	m(S_o + \nh) 
	&= m \bigl( \varphi(S_o + \nh) - \varphi(0) \bigr) 
		&& (\varphi \in \Aut X_o)
	\\&= \alpha_0(S_o + \nh) - \alpha_0(0) 
		&& (\alpha_0(x) = m \, \varphi(x))
	\\&= S_e
	,
		&& \begin{pmatrix} \text{$\alpha_0 = \alpha_1$ is an isomorphism} \\
			\text{from~$X_o$ to~$X_e$} \end{pmatrix}
	\end{align*}
so $m S_o = S_e + \nh$. Therefore
	\[ mS = m(S_e \cup S_o) = mS_e \cup mS_o = (S_o + \nh) \cup (S_e + \nh) = S + \nh . \]
So $X$ has Wilson type~\pref{Wilson-C4}.

\pref{InterchangeCI-squarefree} The order of~$X$ is square-free, so \cref{Muzychuk-CI} tells us that $X$ has the Cayley Isomorphism Property. Therefore~\pref{InterchangeCI-X} applies.

\pref{InterchangeCI-valency} It is known \cite[\S7.2]{Li-IsoSurvey} that every connected, circulant graph of valency \text{$\le 5$} has the Cayley Isomorphism Property. (A proof for valency~$4$ can also be found in~\cite[Thm.~5.4]{MuzychukKlinPoeschel}.) Therefore~\pref{InterchangeCI-Xe} applies.
\end{proof}

\begin{proof}[\bf Proof of \cref{2p}]
Let $X = \Cay(\ZZ_{2p}, S)$ be a nontrivially unstable, circulant graph of order $n = 2p$. 
It is easy to see, by inspection, that there are no nontrivially unstable, circulant graphs of order~$4$, so $p$ is odd. Therefore $n = 2p$ is square-free.

Let 
	\[ S' = S \setminus (S + p) . \]
Since $X$ is twin-free, we know that $S + p \neq S$ (see \cref{TwinFreeIffNoTranslation}), which means that $2S'$ is nonempty.

\setcounter{case}{0}
\
\begin{case}
Assume $2S' \neq \{0\}$.
\end{case}
Since $2 \ZZ_{2p}$ has order~$p$, which is prime, every nonzero element is a generator. So $2S'$ generates $2\ZZ_{2p}$. We also know from \cref{2S'-BX} that every automorphism of $BX$ is an automorphism of 
	\[ \Cay \bigl( \ZZ_{2p} \times \ZZ_2, 2S' \times \{0\} \bigr) . \]
By combining these two facts, we conclude that $2\ZZ_{2p} \times \{0\}$ is a block for the action of $\Aut BX$. 

Therefore, \fullcref{InterchangeCI}{squarefree} applies, so $X$ has Wilson type~\pref{Wilson-C1} or~\pref{Wilson-C4}.
However, $2 \ZZ_{2p} \cong \ZZ_p$ has no nontrivial, proper subgroups, so it is obvious that $X$ does not have Wilson type~\pref{Wilson-C1}. Therefore, it must have Wilson type~\pref{Wilson-C4}, which is exactly what we needed to prove.

\begin{case}
Assume $2S' = \{0\}$.
\end{case}
This means that $S' = \{p\}$, so $p \in S$. 

Since $X$ is unstable, we may let $\alpha$ be an automorphism of~$BX$, such that $\alpha(0,1) = (t,1)$ with $t \neq 0$. For all $x \in \ZZ_n$, we see from \cref{0orbit} that if $|x| = |t|$, then $x \notin S$. Since $p \in S$, this tell us that $|t| \neq 2$. So $|t|$ is either $p$ or~$2p$. Therefore, either $S$ does not contain any element of order~$p$, or $S$~does not contain any element of order~$2p$. However, since $2S' = \{0\}$, we also know that $s + p \in S$ for all $s \in S \setminus \{p\}$. Also note that 
	\[ |s| = p \iff |s + p| = 2p . \]
Putting this together, we conclude that $S = \{p\}$. This contradicts the fact that the nontrivially unstable graph~$X$ must be connected.
\end{proof}

\begin{cor} \label{Describe2p}
Let $X = \Cay(\ZZ_{2p}, S)$ be a circulant graph of order~$2p$, where $p$ is an odd prime, and let $S_e = S \cap 2\ZZ_{2p}$. The graph~$X$ is unstable if and only if either it is trivially unstable, or there exists $m \in \ZZ_{2p}^\times$, such that $m^2 S_e = S_e$, $m S_e \neq S_e$, and $S = S_e \cup \bigl( (n/2) + m S_e \bigr)$.
\end{cor}

\begin{proof}
($\Leftarrow$) If $X$ is not trivially unstable, the conditions imply that $X$ has Wilson type~\pref{Wilson-C4}.

($\Rightarrow$) Assume $X$ is nontrivially unstable. We conclude from \cref{2p} that $X$ has Wilson type~\pref{Wilson-C4}, so there is some $m \in \ZZ_{2p}$, such that $S = m S + p$. Since $p$ is odd, this implies that $S_o = m S_e + p$ (where $S_o = S \setminus S_e$) and
	\[ S_e = m S_o + p = m(m S_e + p) = m^2 S_e . \]
If $m S_e = S_e$, then $S_o = S_e + p$, so $S = S + p$, which contradicts the fact that nontrivially unstable graphs are twin-free.
\end{proof}

\begin{cor} \label{orders}
For $n \in \ZZ^+$, there does \textbf{not} exist a nontrivially unstable circulant graph of order~$n$ if and only if either $n$~is odd, or $n < 8$, or $n = 2p$, for some prime number $p \equiv 3 \pmod{4}$.
\end{cor}

\begin{proof}
($\Rightarrow$)
If $\nh$ is not prime, then $2\ZZ_n \cong \ZZ_{n/2}$ has a nontrivial, proper subgroup~$A$. Choose some $b \in 2\ZZ_n \setminus A$, and let $S = \{\pm1\} \cup (\pm b + A)$, so $S_e := S \cap 2\ZZ_n = \pm b + A$. Then $X = \Cay(\ZZ_n, S)$ has Wilson type~\pref{Wilson-C1}, so it is unstable.

If $\nh$ is prime, and $\nh \not\equiv 3 \pmod{4}$ (and $n \ge 8$), then $\nh \equiv 1 \pmod{4}$, so there exists $m \in \ZZ_n^\times$, such that $m^2 = -1$. Let $S = \{\pm 1, \nh \pm m \}$, so $S_e := S \cap 2\ZZ_n = \{\pm m + \nh\}$ and $S = m S + \nh$. Then $X = \Cay(\ZZ_n, S)$ has Wilson type~\pref{Wilson-C4}, so it is unstable.

In either case, $X$ is also connected (because $1 \in S$) and nonbipartite (because $S_e \neq \emptyset$). 
Hence, if $X$ is not nontrivially unstable, then it must not be twin-free, so there is a nonzero $h \in \ZZ_n$, such that $h + S = S$.
Note that in both cases, $S_o=\{\pm 1\}$ and since $n > 4$ and $h$ is nonzero, it cannot happen that $\{\pm 1\}+h=\{\pm 1\}$. It follows that $S_o+h=S_e$ and $S_e+h=S_o$ (and $h$ is odd).

Since $S_o + 2h = S_o$ (and $S_0 = \{\pm 1\}$), we must have $2h = 0$, which means $h = \nh$, so $S_e = S_o +\nh = \{\nh \pm 1\}$.

If $\nh$ is not prime, then, since $\{\nh \pm 1\} = S_e = \pm b + A$, we must have $\langle 2 \rangle \subseteq A$. Since $n > 4$, this implies $|b + A| = |A| \ge n/2 > 2$, which is a contradiction.

If $\nh$ is prime, we must have $m = \pm 1$ (since $S_e=\{\nh \pm m\}$, which contradicts the fact that $m^2 = -1$.

($\Leftarrow$)
We prove the contrapositive: supposing there does exist a nontrivially unstable circulant graph of order~$n$, we will show that $n$ is odd, that $n \ge 8$, and that $n/2$ is \emph{not} a prime number that is congruent to~$3 \pmod{4}$.

The fact that $n$ is odd is immediate from \cref{OddCirculant}. Also, it is easy to see, by inspection, that there are no nontrivially unstable circulant graphs of order $2$ or~$4$; so $n \ge 6$.

Now suppose $X = \Cay(\ZZ_{2p}, S)$ is a nontrivially unstable circulant graph of order~$2p$, where $p$ is prime, and $p \equiv 3 \pmod{4}$. (This includes the case where $n = 6$.) We will show that this leads to a contradiction.
By \cref{2p}, we know that $X$ has Wilson type~\pref{Wilson-C4}, so there is some $m \in \ZZ_{2p}^\times$, such that $S = mS + \nh$. Write $m = m_o m_2$, where $m_o$ has odd order (as an element of the group $\ZZ_{2p}^\times$), and the order of $m_2$ is a power of~$2$. Since $\ZZ_{2p}^\times$ is cyclic of order $p -1 \equiv 2 \pmod{4}$, there are no elements of order~$4$ in $\ZZ_{2p}^\times$, so $m_2 \in \{\pm 1\}$. Since $S = -S$, this implies $S = m_2 S$, so we conclude that $S = m_oS + \nh$. After repeatedly multiplying both sides of this equation by~$m_o$, we see that $S = m_o^k S + \nh$ for any odd number~$k$, including $k = |m_o|$. Hence, we have $S = S + \nh$. This contradicts the fact that $X$ is twin-free.
\end{proof}

\section{Computational results} \label{ComputationalSect}

S.\,Wilson \cite[p.~377]{Wilson} mentioned: ``There are 3274 circulant graphs which are non-trivially unstable and have no more than 38 vertices.'' However, we performed computations that produce a different number: there seem to be 3576 such graphs (up to isomorphism). The interested reader can reproduce our results in only a few minutes by running the Sagemath\footnote{Sagemath code can be run online at \url{https://cocalc.com/}} code or Magma code or Maple code that is available in the ancillary files directory of this paper 
on the arxiv.

Analysis of these graphs establishes:

\begin{obs} \label{Unstable24}
Every nontrivially unstable circulant graph of order less than 40 has a Wilson type, \emph{except} the following six graphs (up to isomorphism), all of order~24:
	\begin{enumerate}
	\item $\Cay( \ZZ_{24}, \{ \pm2, \pm3, \pm8, \pm9, \pm10 \} )$,
	\item $\Cay( \ZZ_{24}, \{ \pm2, \pm3, \pm8, \pm9, \pm10, 12 \} )$,
	\item $\Cay( \ZZ_{24}, \{ \pm1, \pm2, \pm5, \pm7, \pm8, \pm10, \pm11 \} )$,
	\item $\Cay( \ZZ_{24}, \{ \pm1, \pm2, \pm5, \pm7, \pm8, \pm10, \pm11, 12 \} )$,
	\item $\Cay( \ZZ_{24}, \{ \pm1, \pm2, \pm3, \pm5, \pm7, \pm8, \pm9, \pm10, \pm11 \} )$,
	\item $\Cay( \ZZ_{24}, \{ \pm1, \pm2, \pm3, \pm5, \pm7, \pm8, \pm9, \pm10, \pm11, 12 \} )$.
	\end{enumerate}
The first of these graphs appears explicitly in \cite[p.~156]{QinXiaZhou}, and it seems that some (or all) of the others were also known to the authors of that paper.
\end{obs}

\begin{rem}
Several additional days of computer time extended the exhaustive calculations to order~50. (However, these calculations are certainly not definitive, because they were performed only once, so their correctness has not been verified.) Lists of the additional 67725 nontrivially unstable graphs have been deposited in the ancillary files directory of this paper 
on the arxiv. 

We note that the instability of every example that was found is explained by the instability conditions in \cref{InstabilitySect}. 
On the other hand, the calculations uncovered 316 additional examples of nontrivially unstable circulant graphs with no Wilson type: 
	52 of order~40, 
	262 of order~48,
	and
	2 of order~50.
\end{rem}

\AtEndDocument{\input{NotesToReferee}}  

\end{document}

%% file: NotesToReferee.tex


\newpage
\addtocontents{toc}{\vskip\bigskipamount}
\begin{appendix}

\section{Notes to aid the referee}

\begin{aid} \label{TwinFreeIffNoTranslation-aid}
We prove the contrapositive.

($\Rightarrow$) Suppose there exists $h \in \ZZ_n$, such that $h + S = S$. We claim that $0$ and~$-h$ have exactly the same neighbors, so $\Cay(\ZZ_n, S)$ is not twin-free:
	\begin{align*}
	\text{$v$ is a neighbor of~$0$} 
	& \iff v \in S
	\\&\iff h + v \in S
	\\&\iff v - (-h) \in S
	\\&\iff \text{$v$ is adjacent to~$-h$}
	. \end{align*}

($\Leftarrow$) Suppose $v$ and~$w$ have exactly the same neighbors in $\Cay(\ZZ_n, S)$, and let $h = v - w$ (so $w = v - h$). We claim that $h + S = S$:
	\begin{align*}
	S 
	&= \{\, x - w \mid \text{$x$ is a neighbor of~$w$} \,\}
	\\&= \{\, x - w \mid \text{$x$ is a neighbor of~$v$} \,\}
	\\&= \{\, x - (v - h) \mid \text{$x$ is a neighbor of~$v$} \,\}
	\\&= h + \{\, x - v \mid \text{$x$ is a neighbor of~$v$} \,\}
	\\&= h + S
	. \end{align*}
\end{aid}

\begin{aid} \label{BlockDefn-H}
Let $\hat G$ be the regular representation of $\ZZ_n \times \ZZ_2$. Since $\hat G$ is contained in $\Aut BX$, it is clear that $\mathcal{B}$ is also a block for the action of~$\hat G$. Hence, \cite[Exer.~1.5.6, p.~13]{DixonMortimer} tells us that $\mathcal{B}$ is an orbit of some subgroup~$\hat H$ of~$\hat G$. If we let $H$ be the corresponding subgroup of $\ZZ_n \times \ZZ_2$, then the $\hat H$ orbit of an element~$b$ of $\ZZ_n \times \ZZ_2$ is precisely $b + H$, so it is a coset of~$H$.

We see from \cite[Exer.~1.5.3, p.~13]{DixonMortimer} that every coset of~$H$ is a block. Furthermore, since these cosets form a partition of $\ZZ_n \times \ZZ_2$, we see that they are the entire ``system of blocks'' containing~$\mathcal{B}$, so \cite[Exer.~1.5.3, p.~12]{DixonMortimer} tells us that $\Aut BX$ acts on this set of blocks.
\end{aid}

\begin{aid} \label{WilsonC1C3-inS-aid}
Since $H = \langle h \rangle$ and $|h| = 2$, it is clear that $|H| = 2$.

If $|K|$ is divisible by~$4$, then (since $K$ is a subgroup of~$\ZZ_n$, and is therefore cyclic) we know that $K$ has an element~$k$ of order~$4$. So $|2k| = 2$. Since the element of order~$2$ in~$\ZZ_n$ is unique, we conclude that $h = 2k \in 2K$. This contradicts the fact that $h \in K_o$.
\end{aid}

\begin{aid} \label{WilsonIsSpecialCase-C1-aid}
Since $h \in 2K$, we may write $h = 2k$, for some $k \in K$. If $|h| = 2$, this implies $|k| = 4$, so $|K|$ must be divisible by~$4$.
\end{aid}

\begin{aid} \label{WilsonIsSpecialCase-C3-aid}
If $|H| = 2$, then $H = \langle \nh \rangle$.
Also, since $|K|$ is even, we know that $\nh \in K$. So
	$\nh \in H \cap K \subseteq 2K$,
which means there is some $k \in K$, such that $2k = \nh$.
Then $|k| = 4$, so $|K|$ is divisible by~$4$.
\end{aid}

\begin{aid} \label{IsoTranslateS-aid}
Suppose $uv$ is an edge of~$X$. Since $\alpha$ is an isomorphism from $\Cay(\ZZ_n, S)$ to $\Cay(\ZZ_n, S + \nh)$, this implies that $\alpha(u)$ is adjacent to $\alpha(v)$ in $\Cay(\ZZ_n, S + \nh)$, which means $\alpha(u) - \alpha(v) \in S + \nh$. So
	\[
	\alpha(u) - \beta(v)
	= \alpha(u) - \bigl( \alpha(v) + \nh \bigr)
	= \bigl( \alpha(u) - \alpha(v) \bigr) + \nh
	\in (S + \nh) + \nh
	= S
	. \]
\end{aid}

\begin{aid} \label{IsoTranslateSEg-notWilson-aid}
Since $2\ZZ_n / \langle a \rangle \cong \ZZ_{p^2}/p \ZZ_{p^2} \cong \ZZ_p$, there is a homomorphism $\varphi \colon 2\ZZ_n \to \ZZ_p$, such that $\varphi(2) = 1$ and $\ker \varphi = \langle a \rangle$. Then $\varphi(S_e) = \{\pm 1, 0\}$ and $\varphi(S_o')= \{\pm 1, 0\}$. Therefore, in~$\ZZ_p$, we have
	\[ \{\pm m, 0 \}
	= m \cdot \{\pm 1, 0\}
	= m \cdot \varphi(S_e)
	= \varphi(m S_e)
	= \varphi(S_o')
	= \{\pm 1, 0\} 
	, \]
so $m \equiv \pm 1 \pmod{p}$.
\end{aid}

\begin{aid} \label{val8Eg-aid}
Let $a = 2^{\ell-2} = n/12 \in S$.
	\begin{itemize}
	\item Since $\gcd(3,a) = 1$, we know that $X$ is connected.
	\item Since $3$ is odd and $a$ is even (because $\ell > 2$), we know that $X$ is not bipartite.
	\item Suppose $X$ is not twin-free. Then \cref{TwinFreeIffNoTranslation} provides some nonzero $h \in \ZZ_n$, such that $h + S = S$. Since $|S| = 8$, this implies that $|h|$ is a divisor of~$8$, so $|h|$ is even, which implies that $\nh \in \langle h \rangle$. Therefore, we have $\nh + S = S$. However, the only even elements of~$S$ are $\pm 6$ and $\pm n/12$, so it is clear that $\nh + 6 \notin S$ (because $\ell > 3$). This is a contradiction.
	\end{itemize}
\end{aid}

\begin{aid} \label{val8Eg-unstable-aid} We have
	\begin{align*}
	m (\nh + 3)
	&= \left( \frac{n}{6} - 1 \right) \left( \frac{n}{2} + 3 \right)
	\\&= (2^{\ell-1} - 1) (3 \cdot 2^{\ell -1} + 3 )
	\\&= 3 (2^{\ell-1} - 1)(2^{\ell-1} + 1)
	\\&= 3(2^{2\ell - 2} - 1)
	\\&= 3\cdot 2^{2\ell - 2} - 3
	\\&= 2^{\ell - 2} \, n - 3
	\\&\equiv -3 \qquad \pmod{n}
	.\end{align*}
\end{aid}

\begin{aid} \label{val8Eg-noWilson}
Since $n \ge 3 \cdot 2^4 = 48 > 36$, have
	\[ (\nh - 3) -  \frac{n}{12}
	= \frac{6n - 36 - n}{12}
	> \frac{4n}{12}
	= \frac{n}{3}
	. \]

If $\ell \ge 5$, then $n/12 \ge (3 \cdot 2^5)/12 = 8 > 6$, so $S$ does not contain any elements between $n/12$ and $\nh - 3$. Therefore, the preceding inequality implies that $h > n/3$. Since $h$ is a divisor of~$n$, we conclude that $h = \nh$.

A bit of additional argument is required in the special case where $\ell = 4$ (because $n/12 = 4 < 6$). However, in this case, we have $(\nh - 3) - 6 = 21 - 6 = 15$. This is not a divisor of $48 = n$, so this difference is not equal to~$h$. Therefore, $6$ and $\nh - 3$ cannot be two consecutive elements of a coset of~$H$, so we once again must have $h \ge (\nh - 3) -  (n/12) > n/3$.
\end{aid}

\begin{aid} \label{XeUnstable-aid}
If $u$ and~$v$ are odd, then $\alpha'(u) = u$ and $\beta'(v) = v$, so it is obvious that $\alpha'(u)$ is adjacent to $\beta'(v)$. Also, if $u$ and~$v$ are even, then it is immediate from~\pref{XeUnstable-even} that $\alpha'(u)$ is adjacent to $\beta'(v)$.

So we may assume that $s \coloneqq u - v$ is odd. Then we see from~\pref{XeUnstable-odd} that $s + H \subseteq S$. From~\pref{XeUnstable-H} (and the fact that $\alpha'(x) = \beta'(x) = x \in x + H$ when $x \notin 2 \ZZ_n$), we also know that 
	\[ \text{$\alpha'(x) \in x + H$ and $\beta'(x) \in x + H$, for all $x \in \ZZ_n$.} \]
Therefore
	\[ \alpha'(u) - \beta'(v) \in (u + H) - (v + H) = (u - v) + H = s + H \subseteq S , \]
so $\alpha'(u)$ is adjacent to $\beta'(v)$.
\end{aid}

\begin{aid} \label{2S'-BX-aid}
For $v,w \in V(BX)$, let $\#(v,w)$ be the number of walks of length~$2$ from~$v$ to~$w$ in $BX$. These walks are in one-to-one correspondence with the elements of
	\[ \mathcal{W}_{v,w} \coloneqq \{\, (s_1,s_2) \in S \times S \mid \widehat{s_1} + \widehat{s_2} = w - v \,\} , \]
so 
	\[ \#(v,w) = |\mathcal{W}_{v,w}| . \]
Since $\ZZ_n \times \ZZ_2$ is abelian, we have $\widehat{s_2} + \widehat{s_1} = \widehat{s_1} + \widehat{s_2}$, so the map $\pi \colon (s_1,s_2) \mapsto (s_2, s_1)$ is a permutation of~$\mathcal{W}$. Let
	\[ \mathcal{F}_{v,w} \coloneqq \{\, (s,s) \mid s \in S, \ 2\widehat{s} = w - v \,\}  \]
be the set of fixed points of~$\pi$ in its action on~$\mathcal{W}_{v,w}$.
Since the cardinality of every orbit of~$\pi$ is either $1$ or~$2$, we see that 
	\[ |\mathcal{W}_{v,w}| \equiv |\mathcal{F}_{v,w}| \pmod{2} . \]

We claim that $|\mathcal{F}_{v,w}|$ is either $0$, $1$, or~$2$. (So $|\mathcal{F}_{v,w}|$ is odd if and only if $|\mathcal{F}_{v,w}| = 1$.) To see this, suppose $(s,s)$ and $(t,t)$ are two different elements of~$\mathcal{F}_{v,w}$. This means that $s \neq t$ and $2\widehat{s} = w - v = 2\widehat{t}$, so $2s = 2t$. Then $t - s$ must be an element of order~$2$, so $t = s + \nh$ (because $\nh$ is the only element of order~$2$ in~$\ZZ_n$). This completes the proof of the claim. Furthermore, the argument establishes that
	\[ |\mathcal{F}_{v,w}| = 1 \iff w - v \in 2 \widetilde{S'} . \]
So
	\[ \text{$|\mathcal{F}_{v,w}|$ is odd}
		\quad \iff \quad w - v \in 2 \widetilde{S'} . \]

Combining the above facts implies that 
	\[ \text{$\#(v,w)$ is odd} \quad \iff \quad w - v \in 2 \widetilde{S'} . \]
Since every automorphism of~$BX$ must preserve the value of $\#(v,w)$, and $2 \widetilde{S'} = 2S' \times \{0\}$, this implies that every automorphism of~$BX$ is an automorphism of $\Cay \bigl( \ZZ_n \times \ZZ_2, \ 2S' \times \{0\} \bigr)$.
\end{aid}

\begin{aid} \label{congruentorbit-aid}
Since $t \in pS$, we have $t = ps$, for some $s \in S$. However, we also have $t = p \ell t$. So $p \ell t = ps$. Since multiplication by~$p$ is a bijection on $\ZZ_n \times \ZZ_2$, this implies $\ell t = s \in S$.
\end{aid}

\begin{aid} \label{0orbit-aid}
For all $v,w \in \ZZ_n$, we have:
	\begin{align*}
	 (v,0) &\text{ is adjacent to $(w,1)$ in $BX$} 
	 \\& \iff v - w \in S
	 	&& \begin{pmatrix} \text{$X = \Cay(G,S)$, and} \\ \text{definition of $BX$)} \end{pmatrix}
	 \\& \iff v - w \notin S^c
	 	&& \text{(definition of $X^c$)}
	 \\& \iff \begin{matrix} \text{$(v,0)$ is not adjacent} \\ \text{to $(w,1)$ in $BX^c$} \end{matrix} 
	 	&& \begin{pmatrix} \text{$X^c = \Cay(G,S^c)$, and} \\ \text{definition of $BX^c$)} \end{pmatrix}
	 . \end{align*}
Therefore, if $\alpha$ is an automorphism of~$BX$ that fixes $\ZZ_n \times \{0\}$ and $\ZZ_n \times \{1\}$ (setwise), then $\alpha$ is also an automorphism of~$BX^c$.

Now, let $\alpha$ be any automorphism of~$BX$. Since the bipartition of~$BX$ is unique (because $BX$ is connected), we know that $\alpha$ either fixes the two bipartition sets (setwise) or interchanges them. If it fixes the sets, then the preceding paragraph implies that $\alpha$ is an automorphism of $BX^c$, as desired.

Now, suppose $\alpha$ interchanges the two bipartition sets. Let $\alpha'$ be the composition of~$\alpha$ with the translation by $(0,1)$. Then $\alpha'$ is an automorphism of $BX$ that fixes the two bipartition sets, so it is an automorphism of~$BX^c$. Since $\alpha$ differs from~$\alpha'$ only by a translation (which is an automorphism of~$BX^c$), we conclude that $\alpha$ is an automorphism of~$BX^c$.
\end{aid}

\begin{aid} \label{OddStableIso-aid-mSiso}
Since $\varphi(\ZZ_n \times \{0\}) = \ZZ_n \times \{0\}$, there is a bijection $\alpha \colon \ZZ_n \to \ZZ_n$, such that $\varphi(x,0) = \bigl( \alpha(x), 0 \bigr)$ for all $x \in \ZZ_n$.

Let $m = n + 1$. Then $m$ is even (because $n$ is odd), so $m \widetilde S_1 = mS_1 \times \{0\}$. Also,  for all $x \in \ZZ_n$ we have $nx = 0$, so $mx = x$; therefore $m S_1 = S_1$. Then, since $m = n + 1$ is obviously relatively prime to~$n$, \cref{mSrelprime} tells us that $\varphi$ is an isomorphism from $\Cay(\ZZ_n \times \ZZ_2, S_1 \times \{0\})$ to $\Cay(\ZZ_n \times \ZZ_2, S_2 \times \{0\})$. By the definition of~$\alpha$, this implies that $\alpha$ is an isomorphism from $\Cay(\ZZ_n, S_1)$ to $\Cay(\ZZ_n, S_2)$.

We already know (from the definition of~$\alpha$) that $\varphi(x,0) = \bigl( \alpha(x), 0 \bigr)$ for all $x \in \ZZ_n$.
The fact that $X_2$ is twin-free now implies that we also have $\varphi(x,1) = \bigl( \alpha(x), 1 \bigr)$. To see this, let $N_1$ be the set of neighbors of~$x$ in~$X_1$, so $N_2 \coloneqq \alpha(N_1)$ is the set of neighbors of $\alpha(x)$ in~$X_2$ (since $\alpha$ is an isomorphism). Then the set of neighbors of $(x,1)$ is $N_1 \times \{0\}$ (by the definition of $BX_1$), so, since $\varphi$ is an isomorphism, the set of neighbors of $\varphi(x,1)$ in $BX_2$ is
	\[ \varphi\big(N_1 \times \{0\} \bigr) 
	= \bigl( \alpha(N_1) \times \{0\} \bigr)
	= N_2 \times \{0\} . \]
Thus, if we write $\varphi(x,1) = (y,1)$, then $N_2$ is the set of neighbors of~$y$ in~$X_2$. Since this is the same as the set of neighbors of $\alpha(x)$, and $X_2$ is twin-free, we conclude that $y = \alpha(x)$. So $\varphi(x,1) = (y,1) = \bigl( \alpha(x), 1 \bigr)$, as claimed.
\end{aid}

\begin{aid} \label{2p-alpha-aid}
First note that $h$ must be even, since $h + S_e = S_e$ (and $S_e$ is nonempty, since $X$ is not bipartite). Therefore, $\alpha$ fixes the two sets $\{x \equiv i \pmod{2}\}$ and $\{x \not\equiv i \pmod{2}\}$. Since $\alpha$ is obviously a bijection on each of these sets, we conclude that $\alpha$ is a bijection.

Now, suppose $(x,0)$ is adjacent to $(y,1)$, and let $s = y - x \in S$.
	\begin{itemize}
	\item If $x$ is odd and $y$ is even, then $\alpha$ fixes $(x,0)$ and $(y,1)$, so it is obvious that $\alpha(x,0)$ is adjacent to $\alpha(y,1)$.
	\item If $x$ is even and $y$ is odd, then $\alpha$ translates $(x,0)$ and $(y,1)$ by $(h,0)$. Since this translation is an automorphism of $BX$, it is again obvious that $\alpha(x,0)$ is adjacent to $\alpha(y,1)$.
	\item Finally, suppose $x$ and $y$ have the same parity, which means $s \in S_e$. Then $\alpha$ fixes one of the two vertices, and translates the other by $(h,0)$, so
		\begin{align*}
		 \alpha(y,1) - \alpha(x,0) 
		&= \bigl( (y,1) - (x,0) \bigr) \pm (h,0)
		\in \bigl( S_e \times \{1\} \bigr) \pm (h,0)
		\\&= \bigl( S_e \pm h \bigr) \times \{1\}
		= S_e \times \{1\}
		\subseteq S \times \{1\}
		, \end{align*}
	so $\alpha(x,0)$ is adjacent to $\alpha(y,1)$.
	\end{itemize}
\end{aid}

 \begin{aid} \label{IsoToCirculant-aid}
 Let $A = [\zeta^{ij}]_{0 \le i,j \le n-1}$ be the Vandermonde matrix. Also let 
 	\[ v = [v_0, v_1, \ldots, v_{n-1}] = [v_i]_{i=0}^{n-1} \in \ZZ^n \]
and 
	\[ w = [w_0, w_1, \ldots, w_{n-1}] = [w_i]_{i=0}^{n-1} \in \ZZ^n , \]
where
 	\[ v_i = \begin{cases}
		1 & \text{if $i \in S$}, \\
		0 & \text{otherwise}
		, \end{cases}
	\]
and
 	\[ w_i = \begin{cases}
		1 & \text{if $i \in mS$}, \\
		0 & \text{otherwise}
		. \end{cases}
	\]
Then
	\[ \text{$\displaystyle v A = \left[ \sum_{s \in S} \zeta^{is} \right]_{i=0}^{n-1} \in \CC^n$
	\quad and \quad
	$\displaystyle w A = \left[ \sum_{s \in mS} \zeta^{is} \right]_{i=0}^{n-1} \in \CC^n$}
	. \]
Since $\sum_{s \in S} \zeta^{is} = \sum_{s \in mS} \zeta^{is}$ for all~$i$, we conclude that $vA = wA$. Since $A$ is invertible, this implies $v = w$, which means $S = mS$.
 \end{aid}

\begin{aid} \label{Interchange-generate-aid}
It is obvious from the definition that $\varphi$ fixes $C_e^0$. (Indeed, it acts as the identity on this set.) Since $\nh$ is odd,  we also see that $\varphi$ maps $C_o^1$ to $C_e^1$. Hence, $\varphi$ maps $G_o = C_e^0 \cup C_o^1$ to $C_e^0 \cup C_e^1 = \ZZ_n \times \ZZ_2$ (and is a bijection).

Now, suppose $(x,0)$ is adjacent to $(y,1)$ in $B_o$. This means that $y - x \in S_o$. Then
	\[ \varphi(y,1) - \varphi(x,0) 
	= (y + \nh, 1) - (x,0) 
	= \bigl( (y - x) + \nh, 1 \bigr) 
	\in (S_o + \nh) \times \{1\}
	, \]
so $\varphi(x,0)$ is adjacent to $\varphi(y,1)$ in $\Cay \bigl( 2 \ZZ_n \times \ZZ_2,  (S_o + \nh) \times \{1\} \bigr)$, as desired.
\end{aid}

\begin{aid} \label{Interchange-iso-aid-XetoXo}
\Cref{OddStableIso} provides an isomorphism $\alpha_0 \colon X_e \to X_o$, such that
	\[ \varphi \bigl( \alpha(x,1) \bigr) = \bigl( \alpha_0(x), 1 \bigr) . \]
From the definition of~$\varphi$, we conclude that $\alpha(x,1) = \bigl( \alpha_0(x) + \nh, 1 \bigr)$.
\end{aid}

\begin{aid} \label{Interchange-iso-aid-XotoXe}
Let $\beta(x,i) = \alpha \bigl( \varphi'(x,i) \bigr) - (\nh, 0)$, so $\beta(C_e^0) = C_e^0$ and $\beta(C_e^1) = C_e^1$, which means that $\beta$ is a bijection on~$G_e$. Note that
	\[ \beta(y,1) - \beta(x,0)
	= \alpha \bigl( \varphi'(y, 1) \bigr) - \alpha\bigl( \varphi'(x,0) \bigr)
	= \alpha(y, 1) - \alpha(x + \nh , 0)
	. \]

Suppose $(x,0)$ is adjacent to $(y,1)$ in $BX_o$. Then $y - x \in S_o + \nh$, so $y - (x + \nh) \in S_o \subseteq S$. Since $\alpha \in \Aut BX$, then
	\[ \alpha(y, 1) - \alpha(x + \nh , 0)
	\in S \times \{1\} . \]
Furthermore, since $(y,1) \in C_e^1$ and $(x + \nh, 0) \in C_o^0$, we know that $\alpha(y, 1) \in C_o^1$ and $\alpha(x + \nh, 0) \in C_o^0$, so
	\begin{align*}
	\beta(y,1) - \beta(x,0)
	&= \alpha(y, 1) - \alpha(x + \nh , 0) 
	\\&\in \bigl( S \times \{1\} \bigr) \cap ( C_o^1 - C_o^0) 
	\\&= \bigl( S \times \{1\} \bigr) \cap C_e^1
	\\&= S_e \times \{1\} 
	. \end{align*}
So $\beta$ is an isomorphism from $BX_o$ to $BX_e$.
\end{aid}

\begin{aid} \label{Interchange-iso-alphaprime}
Suppose $v - w \in S_e$. Then $v$ and $w$ have the same parity, so the calculations of $\alpha'(v)$ and $\alpha'(w)$ both use the same formula, so either
	\[ \alpha'(v) - \alpha'(w) 
	= \alpha_0(v) - \alpha_0(w)
	\in S_o + \nh \]
or
	\begin{align*} \alpha'(v) - \alpha'(w) 
	&= \bigl( \alpha_0(v + \nh)  + \nh \bigr) - \bigl( \alpha_0(w + \nh)  + \nh \bigr) 
	\\&=  \alpha_0(v + \nh) - \alpha_0(w + \nh)
	\\&\in S_o + \nh \end{align*}
(because $(v + \nh) - (w + \nh) = v - w \in S_e$).
\end{aid}

\end{appendix}